\documentclass[final]{siamltex}
\usepackage{amsmath, amssymb, amsfonts}
\usepackage{graphicx,color,caption,subcaption}
\usepackage{diagbox}
\usepackage[ruled]{algorithm2e}
\usepackage{algorithmic}
\usepackage{enumitem}
\usepackage{multirow}
\usepackage{epstopdf}
\usepackage{placeins}

\newcommand{\bvec}[1]{\mathbf{#1}}

\newcommand{\Tr}{\mathrm{Tr}}

\newcommand{\vD}{\bvec{D}}
\newcommand{\vr}{\bvec{r}}

\newcommand{\ud}{\,\mathrm{d}}
\newcommand{\ext}{\mathrm{ext}}
\newcommand{\Hxc}{\mathrm{Hxc}}

\newcommand{\abs}[1]{\lvert#1\rvert}
\newcommand{\norm}[1]{\lVert#1\rVert}
\newcommand{\average}[1]{\left\langle#1\right\rangle}

\newcommand{\wt}[1]{\widetilde{#1}}

\newcommand{\ie}{\textit{i.e.}{~}}
\newcommand{\eg}{\textit{e.g.}{~}}
\newcommand{\Or}{\mathcal{O}}

\newcommand{\I}{\mathrm{i}}

\newcommand{\CC}{\mathbb{C}}

\newcommand{\REV}[1]{{#1}}

\numberwithin{equation}{section}
\numberwithin{figure}{section}
\newtheorem{thm}{\protect\theoremname}
\newtheorem{lem}[thm]{\protect\lemmaname}
\newtheorem{rem}[thm]{\protect\remarkname}
\newtheorem{prop}[thm]{\protect\propositionname}
\newtheorem{cor}[thm]{\protect\corollaryname}

\providecommand{\corollaryname}{Corollary}
\providecommand{\lemmaname}{Lemma}
\providecommand{\propositionname}{Proposition}
\providecommand{\remarkname}{Remark}
\providecommand{\theoremname}{Theorem}

\title{Quantum Dynamics with the Parallel Transport Gauge}
\author{
Dong An\thanks{Department of Mathematics, University of California, Berkeley, Berkeley, CA 94720. Email: \texttt{dong\_an@berkeley.edu}}
\and Lin Lin\thanks{Department of Mathematics, University of California, Berkeley, Berkeley, CA 94720 and Computational Research Division, Lawrence Berkeley National Laboratory, Berkeley, CA 94720. Email: \texttt{linlin@math.berkeley.edu}}
}

\begin{document}

\maketitle

\begin{abstract}
  The dynamics of a closed quantum system is often studied with the
  direct evolution of the Schr\"odinger equation.  In this paper, we
  propose that the gauge choice (\ie degrees of freedom irrelevant to
  physical observables) of the Schr\"odinger equation can be generally
  non-optimal for numerical simulation. This can limit, and in some
  cases severely limit the time step size.  We find that
  the optimal gauge choice is given by a parallel transport formulation.
    This parallel transport dynamics can be simply interpreted as the
    dynamics driven by the residual vectors, analogous to those defined in
    eigenvalue problems in the time-independent setup.  The parallel
  transport dynamics can be derived from a Hamiltonian structure, 
  thus suitable to be solved using a symplectic and implicit time
  discretization scheme, such as the implicit midpoint rule, which
  allows the usage of a large time step and ensures the long time
  numerical stability. We analyze the parallel transport dynamics in the
  context of the singularly perturbed linear Schr\"odinger equation, and
  demonstrate its superior performance in the near adiabatic regime.  We
  demonstrate the effectiveness of our method using numerical results
  for linear and nonlinear Schr\"odinger equations, 
  as well as the time-dependent
  density functional theory (TDDFT) calculations for electrons in a
  benzene molecule driven by an ultrashort laser pulse. 
\end{abstract}

\begin{keywords}
  Schr\"odinger equation; Quantum dynamics; Gauge; Parallel transport; 
  Density matrix; von
  Neumann equation; Symplectic method; Singularly perturbed system;
  Time-dependent density functional theory; Adiabatic theorem
\end{keywords}

\section{Introduction}

Consider the following set of coupled nonlinear Schr\"odinger equations
\begin{equation}\label{eqn:SEgeneral}
  \I \epsilon \partial_t \Psi(t) = H(t,P) \Psi(t).\\
\end{equation}
\REV{Here we assume $0<\epsilon \ll 1$.} $\Psi(t)=[\psi_{1}(t),\ldots,\psi_{N}(t)]$ are $N$ time-dependent
 wave functions subject to suitable initial and boundary conditions. 
$H(t,P)$ is a self-adjoint time-dependent Hamiltonian. 
\REV{$P(t)$ is called the density matrix and defined as 
\begin{equation}
  P(t) = \Psi(t)\Psi^{*}(t) = \sum_{j=1}^{N} \psi_{j}(t)\psi_{j}^{*}(t).
  \label{eqn:dm}
\end{equation}}
\REV{Note that when the initial state $\Psi(0)$ consists of $N$ orthonormal functions, the functions in $\Psi(t)$ will remain orthonormal for all $t$, i.e. $(\psi_i(t),\psi_j(t))=\delta_{ij}$, where $(\cdot,\cdot)$ denotes a suitable inner product. Then}
$$
P^2(t)=\sum_{j,k=1}^N \psi_j(t)(\psi_j(t),\psi_k(t))\psi_k^*(t)=\sum_{j=1}^{N} \psi_{j}(t)\psi_{j}^{*}(t)=P(t),
$$
\REV{\ie $P(t)$ is a projector.
}
The explicit dependence of the Hamiltonian
on $t$ is often due to the existence of an external field, and we assume
the partial derivatives $\frac{\partial^{m} H}{\partial t^{m}}$ are
of $\Or(1)$ in some suitable norms for all $m\geq 1$. 
Hence when $0 < \epsilon \ll 1$, the wave functions can oscillate on
a much smaller time scale than that of the external fields, and this is
called the singularly perturbed regime~\cite{HairerWanner1991}.

The equations~\eqref{eqn:SEgeneral} are rather general and appear in
several fields of scientific computation.  In the
simplest setup when $N=1$ and $H(t,P)\equiv H(t)$, this is the linear
Schr\"odinger equation.  Another example is the nonlinear
Schr\"odinger equation (NLSE) used for modeling nonlinear
photonics and Bose-Einstein condensation process
~\cite{FetterWalecka2003},
\begin{equation}\label{eqn:NLSexample}
    \I\epsilon \partial_t \psi(t) = H_0(t)\psi(t) + g|\psi(t)|^2\psi(t){,}
\end{equation}
where $H_0(t)$ is a Hermitian matrix obtained by discretizing the linear 
operator $-\frac{1}{2}\Delta + V(x,t)$. Since $N=1$,
$P(t)=\psi(t)\psi^{*}(t)$, and $|\psi(t)|^2=\mathrm{diag}[P(t)]$ is
a nonlinear local potential.
When $N>1$, the coupled set of Schr\"odinger equations must be solved
simultaneously. This is the case in the time-dependent density
functional theory (TDDFT)~\cite{RungeGross1984,OnidaReiningRubio2002}.  

The simulation of Eq.~\eqref{eqn:SEgeneral} and in
particular~\eqref{eqn:NLSexample} has been studied via a wide
range of numerical discretization
methods, such as explicit Runge-Kutta
methods~\cite{SchleifeDraegerKanaiEtAl2012}, implicit Runge-Kutta
methods~\cite{CastroMarquesRubio2004}, operator splitting 
methods~\cite{BaoJinMarkowich2002,Lubich2008}, Magnus expansion
methods~\cite{CastroMarquesRubio2004,ChenPolizzi2010}, exponential time
differencing methods~\cite{KassamTrefethen2005},
spectral deferred correction methods~\cite{JiaHuang2008}, dynamical low
rank approximation~\cite{KochLubich2007}, adiabatic
state expansion~\cite{JahnkeLubich2003,WangLiWang2015}, to name a few.
What this paper focuses on is not to develop another numerical 
scheme to directly discretize~\eqref{eqn:SEgeneral}, but to propose an
alternative formulation that is equivalent to~\eqref{eqn:SEgeneral}, and
can be solved with improved numerical efficiency using existing
discretization schemes.

More specifically, note that if we multiply $\Psi(t)$ by a
time-dependent unitary matrix $U(t)\in \CC^{N\times N}$, the resulting set of 
rotated wave functions, denoted by $\Phi(t)=\Psi(t)U(t)$, yields the same
density matrix as
\begin{equation}
  P(t)=\Phi(t)\Phi^{*}(t) = \Psi(t)\left[U(t)U^{*}(t)\right]\Psi^{*}(t)
  = \Psi(t)\Psi^{*}(t).
  \label{eqn:dmgauge}
\end{equation}
Since the unitary rotation matrix $U(t)$ is irrelevant to the density
matrix which is used to represent many physical observables, 
$U(t)$ is called the gauge, and Eq.~\eqref{eqn:dmgauge}
indicates the density matrix is \textit{gauge-invariant}.
Furthermore, Eq.~\eqref{eqn:SEgeneral} can be
directly written in terms of the density matrix as 
\begin{equation}
  \I \epsilon \partial_t P(t) = [H(t,P),P(t)],  
  \label{eqn:vonNeumann}
\end{equation}
where $[H,P]:=HP-PH$ is the commutator between $H$ and $P$.
Eq.~\eqref{eqn:vonNeumann} is called the von Neumann equation (or
quantum Liouville equation), which can be viewed as a more intrinsic
representation of quantum dynamics since the gauge degrees of freedom
are eliminated completely. 

The simulation of the von Neumann equation can also be advantageous from
the perspective of time discretization. Consider the simplified scenario
that $H(t,P)\equiv H(P)$ does not explicitly depend on $t$, and the
initial state $\Psi(0)$ consists of a set of eigenfunctions of $H$, \ie
\begin{equation}
  H[P] \psi_{j}(0) = \psi_{j}(0) \lambda_{j}(0), \quad j=1,\ldots,N, \quad
  P=\sum_{j=1}^{N}\psi_{j}(0)\psi_{j}^{*}(0).
  \label{eqn:nleig}
\end{equation}
Eq.~\eqref{eqn:nleig} is a set of nonlinear eigenvalue equations. When
solved self-consistently, the solution to the Schr\"odinger
equation~\eqref{eqn:SEgeneral} has an analytic form
\begin{equation}
  \psi_{j}(t)=\exp\left( -\frac{\I}{\epsilon} \lambda_{j}(0) t \right)
  \psi_{j}(0), \quad j=1,\ldots,N,
  \label{eqn:psiphasefactor}
\end{equation}
which oscillates on the $\Or(\epsilon)$ time scale. Hence many
numerical schemes still need to resolve the dynamics with a time step of 
$\Or(\epsilon)$. On the other hand, the right hand side of the von
Neumann equation vanishes for all $t$, and hence nominally can be
discretized with an arbitrarily large time step! Of course 
one can use techniques such as integration
factors~\cite{CohenJahnkeLorenzEtAl2006} to make this simulation using the Schr\"odinger
equation as efficient. However this example illustrates that the gap 
in terms of the size of the time step generally exists between the
Schr\"odinger representation and the von Neumann representation. 

In this paper, we identify that such gap is solely due to the gauge degrees
of freedom in the Schr\"odinger representation. By optimizing the gauge
choice, one can propagate the wave functions using a time step comparable
to that of the von Neumann equation. We demonstrate that the optimized
gauge is given by a parallel transport (PT) formulation. We refer to this gauge
as the parallel transport gauge, and the resulting dynamics as the
parallel transport dynamics. Correspondingly the trivial gauge 
$U(t)\equiv I_{N}$ in Eq.~\eqref{eqn:SEgeneral} is referred to as the Schr\"odinger
gauge, and the resulting dynamics as the Schr\"odinger dynamics.
We remark that the PT dynamics can also be interpreted as an analytic
and optimal way of performing the dynamical low rank
approximation~\cite{KochLubich2007} for Eq.~\eqref{eqn:SEgeneral}.  Note
that the simulation of the von Neumann equation requires the explicit
operation on the density matrix $P(t)$.  When a large basis set such as
finite elements or planewaves is used to discretize the partial
differential equation, the storage cost of $P(t)$ can be often
prohibitively expensive compared to that of the wave functions $\Psi(t)$.
Hence the PT dynamics combines the advantages of both approaches, namely to
perform simulation using the time step size of the von Neumann equation,
but with cost comparable to that of the Schr\"odinger equation. 

We analyze the effectiveness of the PT dynamics for the linear
time-dependent Schr\"odinger equation in the near adiabatic regime.  We
remark that efficient numerical methods have been recently developed in
this regime based on the construction of a set of instantaneous
adiabatic states~\cite{JahnkeLubich2003,WangLiWang2015}. The assumption
is that the wave functions can be approximated by the subspace spanned by
low energy eigenstates of the Hamiltonian at each $t$. The dimension of
the subspace is often chosen to be $cN$, where $c$ is a relatively small
constant.  Compared to these methods, the PT dynamics always operates
only on $N$ wave functions, and therefore has reduced computational and
the storage cost.  The PT dynamics is also applicable beyond the near
adiabatic regime.

By extending the quantum adiabatic
theorem~\cite{Nenciu1993,AvronElgart1999} to the PT dynamics, we prove that the local truncation
error of the PT dynamics gains an extra order of accuracy in terms of
$\epsilon$, when the time step is $\Or(\epsilon)$ or
smaller. 
The PT dynamics, after a slight modification, can be derived from a Hamiltonian system similar to that
in the Schr\"odinger dynamics. Hence the gain of accuracy for the local
truncation error can be directly translated to the global error as well
for long time simulation.

We demonstrate the effectiveness of the PT dynamics 
using numerical results of the  model linear
and nonlinear Schr\"odinger equations. We also perform time-dependent
density functional theory (TDDFT) calculations for the electrons in a
benzene molecule driven by an ultrashort laser pulse, near and beyond the
adiabatic regime.  
When the spectral radius of the Hamiltonian is large, it is suitable to
discretize the PT dynamics using a symplectic and implicit time
discretization scheme, such as the implicit midpoint rule, and the
resulting scheme can significantly outperform the same scheme for the
Schr\"odinger dynamics. We also find that other time-reversible and
implicit time discretization schemes, such as the Crank-Nicolson scheme,
can yield similar performance as well. 
Numerical results confirm our analysis in the near
adiabatic regime, and indicate that the convergence of the PT dynamics
can start when the time step size is much larger than $\Or(\epsilon)$. This
is in contrast to the Schr\"odinger dynamics where the error stays flat
until the time step reaches below $\Or(\epsilon)$. For TDDFT
calculations, we find that our discretized PT dynamics can achieve 
$31.7$ times speedup in the near adiabatic regime, and $5.3$ times
speedup beyond the adiabatic regime.

This paper is organized as follows. We derive the parallel transport
gauge in Section~\ref{sec:ptgauge}, and discuss the numerical
discretization of the parallel transport dynamics in
Section~\ref{sec:numerical}.  We analyze the parallel
transport dynamics in the singularly perturbed regime in
Section~\ref{sec:analysis}. We then present the numerical results in
Section~\ref{sec:numer}, followed by the conclusion in
Section~\ref{sec:conclusion}.
%

\section{Parallel Transport Gauge}\label{sec:ptgauge}

\REV{Since the concept of the parallel transport gauge is associated with the time propagation instead of spatial discretization, for simplicity of the presentation, unless otherwise specified, we assume that Eq. \eqref{eqn:SEgeneral} represents a discrete, finite dimensional quantum system, i.e. for a given time $t$, $\psi_{j}(t)$ is a finite dimensional vector, and $H(t,P)$ is a finite dimensional matrix. If the quantum system is spatially continuous, we may first find a set of orthonormal bases functions $\{e_j(\vr)\}_{j=1}^{d}$ satisfying $\int e_j^*(\vr)e_{j'}(\vr) \ud \vr = \delta_{jj'}$, and expand the continuous wavefunction as $\wt{\psi}_j(\vr,t) \approx \sum_{j=1}^d \psi_j(t) e_j(\vr)$. Then after a Galerkin projection, Eq. \eqref{eqn:SEgeneral} becomes a $d$-dimensional quantum system, and the inner product for the coefficients $\psi_j(t)$ becomes the standard $\ell^2$-inner product as $(\psi_j(t),\psi_k(t)):=\psi_j^*(t)\psi_k(t)=\delta_{jk}$.  Hence we can use the linear
algebra notation. The star notation is
interpreted as the complex conjugation when applied to a scalar, and
Hermitian conjugation when applied to a vector or a matrix. }

\subsection{Derivation}\label{sec:deriv}

For simplicity let us consider the case $N=1$ first, where the gauge
matrix $U(t)$ simply becomes a phase factor $c(t)\in \CC$,\REV{$|c(t)|=1$}.  Note
that the gauge choice cannot affect physical observables such as the
density matrix.  Hence conceptually we may think that the time-dependent
density matrix $P(t)$ has already been obtained as the solution of the
von Neumann equation~\eqref{eqn:vonNeumann} on some time interval
$[0,T]$. Similarly the wave function $\psi(t)$ satisfying the
Schr\"odinger dynamics is also known.
Then the relation
\begin{equation}
  P(t)\varphi(t) = \varphi(t), \quad \varphi(t) = \psi(t) c(t)
  \label{eqn:projectioncond}
\end{equation}
is satisfied for any gauge choice. 
For simplicity we use the notation 
$\dot{\varphi}(t)=\partial_{t}\varphi(t)$, and drop the explicit
$t$-dependence in all quantities, as well as the $P$-dependence in
the Hamiltonian unless otherwise noted. Our goal is to find
the time-dependent gauge factor $c(t)$ so that the rotated wave function 
$\varphi(t)$ varies \textit{as slowly as possible}.  This gives rise to
the following minimization problem,

\begin{equation}\label{eqn:minproblem}
\begin{split}
  \min_{c(t)} \quad &\norm{\dot{\varphi}(t)}^2_{2} \\
  \text{s.t.} \quad &\varphi(t) = \psi(t)c(t), \quad \abs{c(t)} = 1.
\end{split}
\end{equation}

In order to solve~\eqref{eqn:minproblem}, note that $P(t)$ is a
projector, we split $\dot{\varphi}$ into
two orthogonal components, 
\begin{equation}
  \dot{\varphi} = P\dot{\varphi} + (I-P)\dot{\varphi}.
  \label{}
\end{equation}
By taking the time derivative with respect to both sides of the first equation in 
Eq.~\eqref{eqn:projectioncond}, we have
\begin{equation}
  (I-P)\dot{\varphi} = \dot{P}\varphi.
  \label{eqn:projectionderiv}
\end{equation}
Then
\begin{equation}
  \begin{split}
      \|\dot{\varphi}\|_2^2 &= \|P\dot{\varphi}\|_2^2 +
      \|(I-P)\dot{\varphi}\|_2^2 \\
      &= \|P\dot{\varphi}\|_2^2 + \|\dot{P}\varphi\|_2^2 \\ 
      &= \|P\dot{\varphi}\|_2^2 + \|\dot{P}\psi\|_2^2.
  \end{split}
\end{equation}
In the last equality, we have used that $\abs{c(t)}=1$. Note that the
term $\|\dot{P}\psi\|_2^2$ is independent of the gauge choice, so
$\|\dot{\varphi}\|_2^2$ is minimized when 
\begin{equation}\label{eqn:PTcondition}
    P\dot{\varphi} = 0.
\end{equation}
Therefore instead of writing down the minimizer of
Eq.~\eqref{eqn:minproblem} directly, we define the gauge implicitly
through Eq.~\eqref{eqn:PTcondition}.



Let us write down an equation for $\varphi(t)$ directly. Combining
equations~\eqref{eqn:projectionderiv},~\eqref{eqn:PTcondition},
~\eqref{eqn:vonNeumann} and~\eqref{eqn:projectioncond}, we have
\begin{equation}
  \dot{\varphi} = \dot{P}\varphi 
    = \frac{1}{\I\epsilon}[H,P]\varphi = \frac{1}{\I\epsilon}(H\varphi -
    \varphi(\varphi^*H\varphi)),
\end{equation}
or equivalently
\begin{equation}\label{eqn:PTLSEsingle}
    \I \epsilon \partial_t \varphi =  H\varphi -
    \varphi(\varphi^*H\varphi).
\end{equation}

For reasons that will become clear shortly, we refer to this gauge
choice as the \emph{parallel transport gauge}, and
Eq.~\eqref{eqn:PTLSEsingle} as the parallel transport (PT) dynamics.
Comparing with the Schr\"odinger dynamics, we find that the PT dynamics
only introduces one extra term $\varphi(\varphi^*H\varphi)$. The right
hand side of Eq.~\eqref{eqn:PTLSEsingle} takes the form of the residual
vector in the solution of eigenvalue problem of the
form~\eqref{eqn:nleig}. Hence the PT dynamics can be simply interpreted
as the dynamics driven by the residuals. Therefore we expect that the PT
dynamics can be
particularly advantageous in the \textit{near adiabatic}
regime~\cite{JahnkeLubich2003,WangLiWang2015}, \ie when
$\varphi$ is close to be the eigenstate of $H$, and all the residual vectors
are therefore small.

Now we provide an alternative interpretation of the gauge choice using
the parallel transport formulation associated with a family of
projectors.  For simplicity let us assume
$H(t)$ is already discretized into a finite dimensional Hermitian matrix for each
$t$ and so is $P(t)$.  Given the single parameter family of projectors
$\{P(t)\}$ defined on some interval $[0,T]$, we define
\begin{equation}
    \mathcal{A}(t) = \I \epsilon [\partial_{t} P(t),P(t)].
\end{equation}
It can be directly verified that $\mathcal{A}(t)$ is a Hermitian matrix
for each $t$, and induces a dynamics
\begin{equation}\label{eqn:ptoperator}
  \I \epsilon \partial_t \mathcal{T}(t) = \mathcal{A}(t) \mathcal{T}(t), 
  \quad \mathcal{T}(0)=I.
\end{equation}
$\mathcal{T}(t)$ is a unitary matrix for each
$t$. $\mathcal{T}(t)$ is called the parallel transport evolution
operator (see \eg~\cite{Nakahara2003,CorneanMonacoTeufel2017}).  The connection between the parallel transport 
dynamics and the parallel transport evolution operator is given in
Proposition~\ref{prop:ptrelation}.

\begin{prop}\label{prop:ptrelation}
  Define $\varphi(t)=\mathcal{T}(t)\psi(0)$ where $\mathcal{T}(t)$ is the
  evolution operator satisfying~\eqref{eqn:ptoperator}, and $P(t)$
  satisfies the von Neumann equation~\eqref{eqn:vonNeumann}. Then
  $P(t) = \varphi(t)\varphi^*(t)$, and 
  $\varphi(t)$ satisfies the parallel transport
  dynamics~\eqref{eqn:PTLSEsingle}.
\end{prop}
\begin{proof}
  First we prove the following relation 
  \begin{equation}
    P(t) \mathcal{T}(t) = \mathcal{T}(t) P(0)
    \label{eqn:ptcommute}
  \end{equation}
  by showing that both sides solve the same initial value problem. 
  Note that $\mathcal{T}(t) P(0)$ satisfies a differential equation 
  of the form~\eqref{eqn:ptoperator} 
  with the initial value $\mathcal{T}(0) P(0)$. 
  We would like to derive the differential equation 
  $P(t) \mathcal{T}(t)$ satisfies. 
  Taking the time derivative on both sides of the identity $P(t) = P^2(t)$, we
  yield two useful relations
  \begin{equation}
    \dot{P} = \dot{P} P + P \dot{P}, \quad  P\dot{P}P = 0.
    \label{eqn:Pidentity}
  \end{equation}
  Then using Eq.~\eqref{eqn:ptoperator},
  \[
  \I \epsilon \partial_t(P\mathcal{T}) = \I \epsilon
  \dot{P}\mathcal{T} + \I \epsilon P[\dot{P},P]\mathcal{T} = \I
  \epsilon \dot{P} P\mathcal{T}.
  \]
  On the other hand,
  \[
  \mathcal{A}(P\mathcal{T}) = \I \epsilon (\dot{P} PP\mathcal{T} -
  P\dot{P}P\mathcal{T}) = \I \epsilon \dot{P} P\mathcal{T}.
  \]
  Therefore
  \begin{equation}
    \I \epsilon \partial_t (P\mathcal{T}) = \mathcal{A} (P\mathcal{T}).
    \label{eqn:PTopediff}
  \end{equation}
  Hence $P\mathcal{T}$ also satisfies an equation of the
  form~\eqref{eqn:ptoperator}.  This proves Eq.~\eqref{eqn:ptcommute} by
  noticing further the shared initial condition 
  $P(0)\mathcal{T}(0)=\mathcal{T}(0)P(0)$.

  Using Eq.~\eqref{eqn:ptcommute}, we have
  \begin{equation}
    P(t)\varphi(t) = P(t)\mathcal{T}(t)\psi(0) = \mathcal{T}(t)P(0)\psi(0) = 
    \mathcal{T}(t)\psi(0) = \varphi(t).
    \label{eqn:Pphi1}
  \end{equation}
  Since $\mathcal{T}(t)$ is unitary, we have
  $\norm{\varphi(t)}_{2}=1$ for all $t$. Hence
  \begin{equation}
    P(t) = \varphi(t)\varphi^{*}(t).
    \label{eqn:Pphi2}
  \end{equation}

  The only thing left is to show that the gauge choice in $\varphi(t)$
  is indeed the parallel transport gauge. Using
  Eq.~\eqref{eqn:ptcommute} and~\eqref{eqn:PTopediff}, we have
  \begin{equation}
    \I \epsilon \partial_{t}\varphi = \I \epsilon \partial_{t}
    (\mathcal{T} \psi(0)) = \I \epsilon \partial_{t}(P\mathcal{T}) \psi(0)
    = \I \epsilon \dot{P} P \mathcal{T} \psi(0) = H P \varphi - P H P
    \varphi.
    \label{}
  \end{equation}
  Here we have used the von Neumann equation
  \[
  \I \epsilon \dot{P} = HP-PH.
  \]
  Finally using
  Eq.~\eqref{eqn:Pphi1} and~\eqref{eqn:Pphi2}, we have
  \[
  \I \epsilon \partial_{t}\varphi = H\varphi -
  \varphi(\varphi^{*}H\varphi),
  \]
  which is precisely the parallel transport dynamics.
\end{proof}

In order to see why the parallel transport gauge can be more
advantageous, consider again the time-independent 
example~\eqref{eqn:nleig} in the
introduction for the case $N=1$. We find that the right hand side of
Eq.~\eqref{eqn:PTLSEsingle} vanishes, and the solution is simply
\[
\varphi(t)=\varphi(0)=\psi(0)
\]
for all $t$.  This implies that the parallel transport gauge is
$c(t)=\exp\left( +\frac{\I}{\epsilon} \lambda(0) t \right)$ that
perfectly cancels with the rotating factor in~\eqref{eqn:psiphasefactor}.  Hence the
PT dynamics yields the slowest possible dynamics by completely
eliminating the time-dependent phase factor, and the time step for
propagating the PT dynamics can be chosen to be arbitrarily large as in
the case of the von Neumann equation. 

For a more complex example, consider a time-dependent nonlinear
Schr\"odinger equation in one dimension to be further illustrated in
Section~\ref{sec:numer}.  Fig.~\ref{fig:example_NLS_wave} (a) shows the
evolution of the real part of the solution $\psi(t)$ from the
Schr\"odinger dynamics, and that of $\varphi(t)$ from the PT dynamics,
respectively.  We find that the trajectory of $\varphi(t)$ varies
considerably slower than that of $\psi(t)$, which allows us to use a
much larger time step for the simulation.  
Fig.~\ref{fig:example_NLS_wave} (b) measures the accuracy of the average
of the orbital center $\average{x}(t)$, using
simulation with the implicit midpoint
rule, also known as the Gauss-Legendre method of order 2 (GL2) scheme.
We compare the performance of the GL2 scheme with the Schr\"odinger gauge
(S-GL2) and that with the PT gauge (PT-GL2) with the same step size
$h=0.004$, and the reference solution is obtained using a very small
step size $h=10^{-5}$.
We observe that the solution from PT-GL2 agrees very well with
the reference solution, while the phase error of the solution from S-GL2
becomes noticeable already after $t=0.2$.

\begin{figure} \centering
    \begin{subfigure}[b]{0.48\textwidth}
        \includegraphics[width=\textwidth]{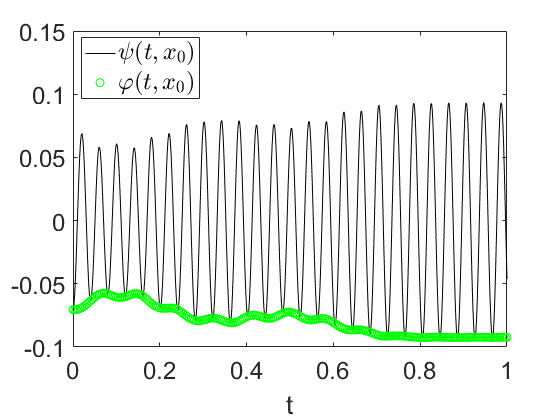}
        \caption{wave functions}
        \label{fig:example_NLS_wave:a}
    \end{subfigure} 
    \begin{subfigure}[b]{0.48\textwidth}
        \includegraphics[width=\textwidth]{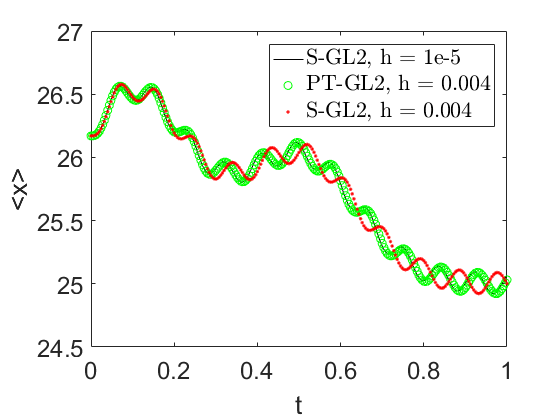}
        \caption{Evolution of centers}
        \label{fig:example_NLS_wave:b}
    \end{subfigure}
    \caption{(a) Real parts of the wave functions at $x_0 = 25$ 
    with the Schr\"odinger gauge and the PT gauge, respectively. 
    (b) Centers of the wave functions.
    Parameters are chosen to be $T = 1, \epsilon = 0.005$, 
    and the reference solution is obtained from S-GL2 with 
    time step size $h = 10^{-5}$. }
    \label{fig:example_NLS_wave}
\end{figure}

\subsection{Hamiltonian structure}\label{sec:hamiltonian}

For simplicity let us consider the linear Schr\"odinger equation, \ie
$H(t,P)\equiv H(t)$, and assume $H(t)$ is a real symmetric matrix for all
$t$. It is well known that the Schr\"odinger dynamics is a Hamiltonian
system~\cite{NettesheimBornemannSchmidtSchutte1996,NettesheimSchutte1999,HairerLubichWanner2006}. More specifically, we separate
the solution $\psi$ into its real and imaginary parts as
\begin{equation}
\psi = q + \I p.
  \label{eqn:psiri}
\end{equation}
\REV{The $\ell^2$-inner product associated with real quantities such as $p,q$ are denoted by $(p,q):=p^Tq$.}
We also introduce the canonically conjugate pair of variables
$(\tau,E)$ to eliminate the explicit dependence of
$H(t)$ on time~\cite{CandyRozmus1991,HairerLubichWanner2006}. This gives the following energy functional
\begin{equation}\label{eqn:LSEenergy}
    \mathcal{E}(\tau,q,E,p) = \frac{1}{2\epsilon}\left[q^TH(\tau)q +
    p^TH(\tau)p\right] + E.
\end{equation}
The Hamiltonian system corresponding to this energy functional is
\begin{equation}\label{eqn:LSEHamiltonian}
  \begin{split}
    \partial_t \tau &= \frac{\partial \mathcal{E}}{\partial E} = 1 {,} \\
    \partial_{t} q &= \frac{\partial
    \mathcal{E}}{\partial p} = \frac{1}{\epsilon}H(\tau)p,\\
    \partial_t E &= -\frac{\partial \mathcal{E}}{\partial \tau} =
    -\frac{1}{2\epsilon}\left[q^T \frac{\partial H(\tau)}{\partial \tau}q +
    p^T \frac{\partial H(\tau)}{\partial \tau} p\right]{,}\\
    \partial_{t} p &= -\frac{\partial
    \mathcal{E}}{\partial q} = -\frac{1}{\epsilon}H(\tau)q {.}
  \end{split}
\end{equation}
Hence $\tau$ is simply the time variable, and $-E$ is the usually
defined energy of the system up to a constant. By
combining the equations for $q,p$ we obtain the Schr\"odinger dynamics
for $\psi$.

Although the PT dynamics only differs from the
Schr\"odinger dynamics by the choice of the gauge, interestingly, the PT dynamics
cannot be directly written as a Hamiltonian system. To illustrate
this, we first separate the real and imaginary parts of $\varphi$ as
in~\eqref{eqn:psiri}, and the PT dynamics can be written as
\begin{equation}\label{eqn:PTReal}
    \begin{split}
  \partial_{t} q =& \frac{1}{\epsilon}(H p - (q^T H q + p^T H p) p),\\
\partial_{t} p =& \frac{1}{\epsilon}(-H q + (q^T H q + p^T H p) q).
\end{split}
\end{equation}
If this dynamics can be derived from some energy functional
$\mathcal{E}$, then
\begin{equation}\label{eqn:PTReal2}
    \begin{split}
   \frac{\partial
  \mathcal{E}}{\partial p} &=\frac{1}{\epsilon}( H p - (q^T H q + p^T H p) p),\\
  \frac{\partial
  \mathcal{E}}{\partial q} &= \frac{1}{\epsilon}(H q - (q^T H q + p^T H p) q).
\end{split}
\end{equation}
Straightforward computation reveals that 
$\frac{\partial^2 \mathcal{E}}{\partial p \partial q} =
\frac{\partial^2 \mathcal{E}}{\partial q \partial p} $
is not true in general, and hence 
the PT dynamics~\eqref{eqn:PTLSEsingle} cannot be a Hamiltonian system.

Fortunately, the PT dynamics can be slightly modified to become a
Hamiltonian system.
Consider the following modified energy functional
\begin{equation}\label{eqn:HamiltonianFunctional}
    \mathcal{E}(\tau,q,E,p) = \frac{1}{2\epsilon} (q^T H(\tau) q + p^T H(\tau) p)(2-q^T q-p^T p) + E.
\end{equation}
The corresponding Hamiltonian equations are
\begin{equation}\label{eqn:PTHamiltonianFor}
    \begin{split}
    \partial_t \tau &= \frac{\partial \mathcal{E}}{\partial E} = 1 {,} \\
  \partial_{t} q &= \frac{\partial
  \mathcal{E}}{\partial p} = \frac{1}{\epsilon}\left[H(\tau) p (2-q^T q-p^T
  p) - (q^T H(\tau) q + p^T H(\tau) p) p \right],\\
  \partial_t E &= -\frac{\partial \mathcal{E}}{\partial \tau} {,}\\
\partial_{t} p &= -\frac{\partial
  \mathcal{E}}{\partial q} = \frac{1}{\epsilon}\left[-H(\tau) q (2-q^T q-p^T
  p) + (q^T H(\tau) q + p^T H(\tau) p) q \right].
\end{split}
\end{equation}
Again $\tau$ is the same as $t$, and the conjugate variable $E(t)$ satisfies 
\[
E(t) = - \frac{1}{2\epsilon} (q^T H(t) q + p^T H(t) p)(2-q^T q-p^T p)
+ \text{constant}.
\]
Compared to the PT dynamics~\eqref{eqn:PTReal}, 
we have an extra factor $(2-q^Tq-p^Tp)$ in the
equations and the energy. Proposition~\ref{prop:PTHamiltonian} states
that the solution to the PT dynamics~\eqref{eqn:PTReal} 
is the same as the solution of 
the Hamiltonian system~\eqref{eqn:PTHamiltonianFor}. 
\begin{prop}\label{prop:PTHamiltonian}
    If $(\tau, q, E, p)$ solves the Hamiltonian system
     \eqref{eqn:PTHamiltonianFor} with normalized initial value condition
    $p^T(0)p(0) + q^T(0)q(0) = 1$,
    then $(q(t),p(t))$ solves \eqref{eqn:PTReal}
    with the same initial value condition,
    and $\varphi(t) = q(t) + \I p(t)$ solves the PT dynamics~\eqref{eqn:PTLSEsingle}.
\end{prop}
\begin{proof}
    Comparing Eq. \eqref{eqn:PTHamiltonianFor} with Eq. \eqref{eqn:PTReal},
    we only need to show the identity
    $$p^Tp + q^Tq = 1$$ holds for all $t$.
    By computing
    \begin{align*}
        \frac{d}{dt} (p^Tp+q^Tq) = &2 (p^T \partial_{t} p + q^T\partial_{t}q) \\
        =&\frac{1}{\epsilon} (-2(2-q^Tq-p^Tp)p^THq + 2(q^THq+p^THp)p^Tq \\
        &+ 2(2-q^Tq-p^Tp)q^THp - 2(q^THq+p^THp)q^Tp) = 0 {,}
    \end{align*}
    we find that $p^Tp+q^Tq$ is invariant during the propagation.
    Together with the normalized initial condition, we complete the proof.
\end{proof}

Proposition~\ref{prop:PTHamiltonian} suggests that the Hamiltonian form
of the PT dynamics is
\begin{equation}\label{eqn:PTLSEsingerH}
    \I \epsilon \partial_t \varphi = H\varphi(2-\varphi^*\varphi) -
  \varphi(\varphi^{*}H\varphi) {,}
\end{equation}
which shares exactly the same solution with~\eqref{eqn:PTLSEsingle}
using the condition $\varphi^{*}\varphi=1$. 


At the end of this part, we briefly discuss the Hamiltonian
structure of the nonlinear Schr\"odinger equation and the associated PT dynamics. 
Let us consider the discretized nonlinear Schr\"odinger
equation~\eqref{eqn:NLSexample}, which can be reformulated as a
Hamiltonian system driven by the energy functional
\begin{equation}
    \mathcal{E}(\tau,q,E,p) = \frac{1}{2\epsilon}\left[q^T H_0(\tau)q +
    p^T H_0(\tau)p + \frac{g}{2}\text{Tr}((|q|^2+|p|^2)^2)\right] + E.
\end{equation}

The PT dynamics corresponding to Eq.~\eqref{eqn:NLSexample} 
can be written as
\begin{equation}\label{eqn:PTNLSexample}
    \I\epsilon \partial_t \varphi = H_0\varphi + g|\varphi|^2\varphi 
    - \varphi(\varphi^*H_0\varphi) - g\varphi(\varphi^*|\varphi|^2\varphi) {.}
\end{equation}
Similar to the linear case, 
the PT dynamics itself cannot be reformulated 
as a Hamiltonian system in general, 
but can be slightly modified to become a Hamiltonian system. 
More precisely, define the energy functional 
\begin{equation}
\begin{split}
\mathcal{E}(\tau,q,E,p) &= \frac{1}{2\epsilon}\left[q^T H_0(\tau)q +
    p^T H_0(\tau)p 
     + g\text{Tr}((|q|^2+|p|^2)^2)\right](2-q^Tq-p^Tp) \\
     &\quad - \frac{g}{4\epsilon}\text{Tr}((|q|^2+|p|^2)^2) + E{,}
\end{split}
\end{equation}
then the Hamiltonian system driven by this energy functional can be
written as
\begin{equation}\label{eqn:PTNLSexampleHam}
    \I\epsilon \partial_t \varphi
    = (H_0\varphi + 2g|\varphi|^2\varphi)(2-\varphi^*\varphi) 
    - \varphi(\varphi^*H_0\varphi) - g\varphi(\varphi^*|\varphi|^2\varphi)
    - g|\varphi|^2\varphi {.}
\end{equation}
Again this equation shares the same solution with Eq.~\eqref{eqn:PTNLSexample} 
using the condition $\varphi^*\varphi = 1$. 

\subsection{General case}\label{sec:generalderiv}


The PT dynamics derived in the previous sections can be directly
generalized to Eq.~\eqref{eqn:SEgeneral} with $N>1$. Define the
transformed set of wave functions
$\Phi(t)=\Psi(t)U(t)=[\varphi_{1}(t),\ldots,\varphi_{N}(t)]$, where $U(t)\in
\CC^{N\times N}$ is a gauge matrix. 
Following the same derivation in Section~\ref{sec:deriv}, we find that
the parallel transport gauge is given by the condition
\begin{equation}
  P\dot{\Phi} = 0.
  \label{}
\end{equation}
This gives rise to the following PT dynamics
%
\begin{equation}\label{eqn:PTNLSmulti}
    \I \epsilon \partial_t \Phi(t)
    = H(t,P(t))\Phi(t) - \Phi(t)[\Phi^*(t)H(t,P(t))\Phi(t)], \quad 
    P(t) = \Phi(t) \Phi^*(t) {.}
\end{equation}
Again the PT dynamics is driven by the residual vectors as 
in eigenvalue problems.

In addition, the Hamiltonian structure is also preserved for the
PT dynamics. For simplicity let us consider the linear Hamiltonian
$H(t)$.  We separate the set of PT wave functions $\Phi$ 
into real and imaginary parts as 
$$\Phi(t) = \mathfrak{q}(t) + \I \mathfrak{p}(t) {.}$$
Define the energy functional
\begin{equation}
    \mathcal{E}(\tau,\mathfrak{q},E,\mathfrak{p}) =
    \frac{1}{2\epsilon}\text{Tr}\Big((\mathfrak{q}^TH(\tau)\mathfrak{q} + \mathfrak{p}^TH(\tau)\mathfrak{p})(2I_N - \mathfrak{q}^T\mathfrak{q} - \mathfrak{p}^T\mathfrak{p})\Big) + E {.}
\end{equation}
The associated Hamiltonian system is
\begin{equation}\label{eqn:PTHamiltonianmulti}
    \begin{split}
        \begin{split}
    \partial_t \tau &= \frac{\partial \mathcal{E}}{\partial E} = 1 {,} \\
  \partial_{t} \mathfrak{q} &= \frac{\partial
  \mathcal{E}}{\partial \mathfrak{p}} = \frac{1}{\epsilon}(H(\tau) \mathfrak{p}(2I_N-\mathfrak{q}^T \mathfrak{q}-\mathfrak{p}^T \mathfrak{p}) - \mathfrak{p}(\mathfrak{q}^T H(\tau) \mathfrak{q} + \mathfrak{p}^T H(\tau) \mathfrak{p}) ),\\
  \partial_t E &= -\frac{\partial \mathcal{E}}{\partial \tau} {,}\\
\partial_{t} \mathfrak{p} &= -\frac{\partial
  \mathcal{E}}{\partial \mathfrak{q}} = \frac{1}{\epsilon}(-H(\tau) \mathfrak{q}(2I_N-\mathfrak{q}^T \mathfrak{q}-\mathfrak{p}^T \mathfrak{p}) +  \mathfrak{q}(\mathfrak{q}^T H(\tau) \mathfrak{q} + \mathfrak{p}^T H(\tau) \mathfrak{p}) ).
\end{split}
    \end{split}
\end{equation}
Similar with the case when $N=1$ (Proposition \ref{prop:PTHamiltonian}),
we can show that
$$\mathfrak{p}^T\mathfrak{p} + \mathfrak{q}^T\mathfrak{q} = I_{N}$$
provided the orthonormal initial value condition.
Therefore the solution to the Hamiltonian system \eqref{eqn:PTHamiltonianmulti}
can exactly form a set of solutions to the PT dynamics.

Due to the straightforward generalization as described above,  unless
otherwise noted, we will focus on the case $N=1$ for the rest of the
paper.

\section{Time discretization}\label{sec:numerical}

When the spectral radius of the Hamiltonian is relatively small and
$\epsilon\sim \Or(1)$, explicit time integrators such as the 4th order
Runge-Kutta method (RK4) and the Strang splitting method can be very efficient,
and can be applied to both the Schr\"odinger dynamics and the PT 
dynamics.
However, the advantage of propagating the PT dynamics can become clearer when $\epsilon$ becomes
small or when the spectral radius of $H$ becomes very large, 
which is typical in \eg TDDFT calculations.
In this
scenario, all explicit
time integrators must take a very small time step, which may become
very costly.  It should be noted that in the
Schr\"odinger dynamics, the solution often oscillates rapidly on the
time scale of $\epsilon$ as indicated in Eq.~\eqref{eqn:psiphasefactor}.  
Standard implicit discretization schemes, such as the implicit midpoint
rule and the Crank-Nicolson scheme, aim at interpolating such rapidly moving
curves by low order polynomials. Therefore the time step must still be
kept on the order of $\epsilon$ to meet the accuracy requirement,
even though the numerical scheme itself may have a large stability
region or even A-stable~\cite{HairerNorsettWanner1987}.

On the other hand, as discussed in Section~\ref{sec:deriv}, 
the PT dynamics transforms the fast oscillating
wave function $\psi(t)$ into a potentially slowly
oscillating wave function $\varphi(t)$ (as in
Fig.~\ref{fig:example_NLS_wave} (a)). 
This makes it feasible to approximate $\varphi(t)$ using a low order
polynomial approximation.  This statement will be further quantified by
numerical results in Section~\ref{sec:numer}. Combined with an implicit
time discretization scheme with a large stability region, we may expect
that the PT dynamics can be discretized with a much larger time step
than that in the Schr\"odinger dynamics.

The Hamiltonian structure of the PT dynamics further invites the usage
of a symplectic scheme for achieving long time accuracy and stability.
The simplest symplectic and implicit scheme is the implicit mid-point
rule, also known as the Gauss-Legendre method of order 2
(GL2). We use a uniform time discretization $t_{n}=n h$, and $h$ is the
time step size.  
With some abuse of notations, we denote by $\varphi(t_{n})$ the exact
solution at $t_{n}$, and $\varphi_{n}$ the numerical approximation to
$\varphi(t_{n})$.  Correspondingly we define 
\[
P_{n} = \varphi_{n}\varphi_{n}^{*}, \quad H_{n}=H(t_{n},P_{n}).
\]
\REV{It would also be helpful to define the effective nonlinear Hamiltonian $H^e(t,\varphi)$ as 
\begin{align*}
    H^e &= H(2-\varphi^*\varphi)-(\varphi^*H\varphi) I, \text{ for Eq.~\eqref{eqn:PTLSEsingerH}, }\\
    H^e &= (H_0+ 2g|\varphi|^2)(2-\varphi^*\varphi)-(\varphi^*H_0\varphi) I - g(\varphi^*|\varphi|^2\varphi) I - g|\varphi|^2, \text{ for Eq.~\eqref{eqn:PTNLSexampleHam}. }
\end{align*}
Then the Hamiltonian equations~\eqref{eqn:PTLSEsingerH} and~\eqref{eqn:PTNLSexampleHam} can be written in a uniform form 
\begin{equation}\label{PTHamUniform}
    \I \epsilon \partial_t\varphi = H^e\varphi{.}
\end{equation}
}

The PT-Ham-GL2 discretization for discretizing the Hamiltonian
equation~\eqref{eqn:PTLSEsingerH} \REV{and~\eqref{eqn:PTNLSexampleHam}} therefore becomes \REV{
\begin{equation}\label{eqn:PTHamiltonian-GL2}
    \begin{split}
    \varphi_{n+1} &= \varphi_n + \frac{h}{\I \epsilon}
    H^e_{n+\frac{1}{2}}\wt{\varphi}
    {,}\\
    \wt{\varphi} &= \frac{1}{2}(\varphi_n + \varphi_{n+1}){,} 
    \end{split}
\end{equation}
Here $\wt{\varphi}$ can be interpreted as the approximation to
$\varphi(t_{n+\frac12})$ at the half time step, and}
\[
    H_{n+\frac{1}{2}}^e:= H^e(t_{n+\frac{1}{2}},\wt{\varphi}).
\]
Note that the normalization condition $\wt{\varphi}^*\wt{\varphi} \to 1$
holds only in the limit $h \to 0$, but $\wt{\varphi}^*\wt{\varphi} \neq
1$ in general. 
Eq.~\eqref{eqn:PTHamiltonian-GL2} is a set of nonlinear equations for
$\varphi_{n+1}$, and need to be solved iteratively. This can be viewed
as a fixed point problem of the form
\[
\varphi = \mathfrak{F}(\varphi),
\]
\REV{where the mapping $\mathfrak{F}$ is explicitly defined as}
\begin{equation}
  \mathfrak{F}(\varphi) = \varphi_n + \frac{h}{\I \epsilon}
H_{n+\frac{1}{2}}^e\wt{\varphi},
    \quad \wt{\varphi} = \frac{1}{2}(\varphi_n + \varphi).
  \label{}
\end{equation}
Assuming the fixed point exists and is unique, we may associate
$\varphi_{n+1}$ with the fixed point, and then move to the next time
step. We may use any nonlinear equation solving technique to solve such
fixed point problem~\cite{Kelley1999}.  
In this work, we use the Anderson
mixing~\cite{Anderson1965} method, which is a simplified Broyden-type
method widely used in electronic structure
calculations~\cite{LinYang2013}.

\REV{The PT-Ham-GL2 scheme can be simplified by directly applying the GL2
discretization to the PT dynamics~\eqref{eqn:PTLSEsingle} and~\eqref{eqn:PTNLSexample}, with the efficient Hamiltonians to be 
defined as 
\begin{align*}
    H^e &= H-(\varphi^*H\varphi) I, \text{ for Eq.~\eqref{eqn:PTLSEsingle}, }\\
    H^e &= H_0+g|\varphi|^2-(\varphi^*H_0\varphi) I - g(\varphi^*|\varphi|^2\varphi) I, \text{ for Eq.~\eqref{eqn:PTNLSexample}.}
\end{align*}
Again note that, unlike the continuous case, PT-GL2 is not equivalent 
to PT-Ham-GL2 since $\wt{\varphi}^*\wt{\varphi} \neq 1$ in general. 
Nevertheless, the norm of the numerical solutions obtained by GL2 
at the discretized time points $t_n$ are indeed conserved, 
which is summarized in the following proposition.
\begin{prop}
    Suppose $\varphi_n$ is the numerical solution obtained by 
    applying GL2 to one of the following PT dynamics,~\eqref{eqn:PTLSEsingerH},~\eqref{eqn:PTNLSexampleHam},~\eqref{eqn:PTLSEsingle} and~\eqref{eqn:PTNLSexample}. Assume that $I - \frac{h}{2\I\epsilon}H^e_{n+\frac{1}{2}}$ is always invertible in each step, then $\|\varphi_n\|_2 = \|\varphi_0\|_2$.
\end{prop}
\begin{proof}
   We consider the GL2 scheme~\eqref{eqn:PTHamiltonian-GL2} for the 
   uniform form~\eqref{PTHamUniform}. It suffices to prove that 
   $\|\varphi_{n+1}\|_2 = \|\varphi_n\|_2$ for any $n$. 
   We first substitute $\wt{\varphi}$ by $\frac{1}{2}(\varphi_n+\varphi_{n+1})$ and rewrite GL2 as 
   \[
       \left(I - \frac{h}{2\I\epsilon}H_{n+\frac{1}{2}}^e\right)\varphi_{n+1} = \left(I + \frac{h}{2\I\epsilon}H_{n+\frac{1}{2}}^e\right)\varphi_{n}.
   \]
   Note that for all defined $H^e$, ${H^e}^{*} = H^e$, then 
   \begin{align*}
       &\varphi_{n+1}^*\varphi_{n+1} \\
       = &\varphi_n^*\left(I + \frac{h}{2\I\epsilon}H_{n+\frac{1}{2}}^e\right)^*{\left(I - \frac{h}{2\I\epsilon}H_{n+\frac{1}{2}}^e\right)^*}^{-1}\left(I - \frac{h}{2\I\epsilon}H_{n+\frac{1}{2}}^e\right)^{-1}\left(I + \frac{h}{2\I\epsilon}H_{n+\frac{1}{2}}^e\right)\varphi_n \\
       = &\varphi_n^*\left(I - \frac{h}{2\I\epsilon}H_{n+\frac{1}{2}}^e\right){\left(I + \frac{h}{2\I\epsilon}H_{n+\frac{1}{2}}^e\right)}^{-1}\left(I - \frac{h}{2\I\epsilon}H_{n+\frac{1}{2}}^e\right)^{-1}\left(I + \frac{h}{2\I\epsilon}H_{n+\frac{1}{2}}^e\right)\varphi_n \\
       = &\varphi_n^*\left(I - \frac{h}{2\I\epsilon}H_{n+\frac{1}{2}}^e\right){\left(I - \frac{h}{2\I\epsilon}H_{n+\frac{1}{2}}^e\right)}^{-1}\left(I + \frac{h}{2\I\epsilon}H_{n+\frac{1}{2}}^e\right)^{-1}\left(I + \frac{h}{2\I\epsilon}H_{n+\frac{1}{2}}^e\right)\varphi_n \\
       = &\varphi_n^*\varphi_n{,}
   \end{align*}
   where the second to the last line uses the fact that $I - \frac{h}{2\I\epsilon}H_{n+\frac{1}{2}}^e$ and $I + \frac{h}{2\I\epsilon}H_{n+\frac{1}{2}}^e$ commute. 
\end{proof}
}

Similarly we may use other time-reversible (but not symplectic) 
schemes~\cite{HairerLubichWanner2006},
such as the trapezoidal rule discretization (known in this context as
the Crank-Nicolson method).  So the PT-CN scheme becomes\REV{
\begin{equation}\label{eqn:PT-CN}
    \varphi_{n+1} = \varphi_n + \frac{h}{2\I \epsilon}
    H_n^e\varphi_n
     + \frac{h}{2\I \epsilon}
    H_{n+1}^e\varphi_{n+1}{,}
\end{equation}
Here $H_n^e = H^e(t_n,\varphi_n), H_{n+1}^e = H^e(t_{n+1},\varphi_{n+1})$. }In both PT-GL2
and PT-CN schemes, we need to solve $\varphi_{n+1}$ with nonlinear
equation solvers as before.  Although these schemes are not symplectic schemes \REV{and the 2-norm of the numerical solution by PT-CN is not strictly conserved} as
in PT-Ham-GL2, numerical results in Section~\ref{sec:numer} 
indicate that the performance of all the
three schemes can be very comparable in practice.

Following the discussion above, we may readily obtain the corresponding
scheme for $N>1$ case, as well as higher order and symplectic time
discretization schemes, such as the Gauss-Legendre collocation
methods~\cite{Iserles2009} for the PT dynamics.

\section{Analysis in the near adiabatic regime}\label{sec:analysis}

In this section, we demonstrate the advantage of the PT dynamics 
by analyzing the accuracy of the
discretized PT dynamics in the near adiabatic regime. Our main result is
that for $h\leq \Or(\epsilon)$, a proper discretization of the PT
dynamics gains one extra order of accuracy in $\epsilon$ compared to
that of the Schr\"odinger dynamics.

We extend the quantum adiabatic theorem~\cite{Kato1950,
AvronElgart1999,Teufel2003} to the PT dynamics, which shows that the PT
wave function $\varphi(t)$ can be decomposed into a component of which
the oscillation is independent of $\epsilon$ and the magnitude is $\Or(1)$, and a component that
is highly oscillatory with $\Or(\epsilon)$ magnitude.  This leads to the
desired result in terms of the local truncation error. We then obtain
the global error estimate from the standard results of symplectic
integrators due to the Hamiltonian structure of the dynamics. 


Again, we restrict the scope of the theoretical analysis to the
time-dependent linear system with $N=1$.  While the generalization to
the case $N>1$ is straightforward, the analysis beyond the linear system
can be considerably more difficult. One important difficulty is 
the lack of the spectral theory \REV{and the corresponding adiabatic theorem} for general nonlinear operators~\cite{Sparber2016}, \REV{which play important roles as being shown in our proof, }
though progress has been made in recent years for certain types of
the nonlinear problems such as the Schr\"odinger equation with weak
nonlinearity~\cite{Sparber2016}, and certain quantum-classical molecular dynamics
(QCMD) models~\cite{BornemannSchutte1999}. 
\REV{We remark that there has been recent progress~\cite{KammererJoye2019} proving the adiabatic theorem under a more general nonlinear setting. Extension of the work of ~\cite{KammererJoye2019} to the nonlinear PT dynamics will be our future work.}


%

We make the following assumptions through this section, 
which defines the near adiabatic regime:
\begin{enumerate}
    \item $H: [0,T] \rightarrow \CC^{d\times d}$ is a Hermitian-valued and
      smooth map. The norms $\norm{H(t)}_{2}$ and
      $\norm{H^{(k)}(t)}_{2}$ for all the
      time derivatives are bounded independently of
      $\epsilon$ and $t\in [0,T]$.
    \item There exists a continuous function 
    $\lambda(t) \in \text{spec}(H(t))$
    which is a simple eigenvalue of $H(t)$ and 
    stays separated from the rest of the spectrum, \ie
    there exists a positive constant $\Delta$ such that
    \begin{equation}\label{}
        \text{dist}(\lambda(t), \text{spec}(H(t)) \backslash \{\lambda(t)\})
        \geq \Delta,
        \quad \forall t \in [0,T] {.}
    \end{equation}
  \item The initial state $\varphi(0)$ is the normalized eigenvector of
    $H(0)$ associated with the eigenvalue $\lambda(0)$.
\end{enumerate}
The assumption 1 ensures that the solutions of both the
Schr\"odinger dynamics and the PT dynamics are smooth with respect to
$t$.
The assumption 2 is called the gap condition~\cite{Teufel2003}. 


Before we continue, 
we would like to investigate a useful conclusion which can be 
directly derived from the above assumptions. 
Let $Q(t)$ denote the projector on the eigenspace
corresponding to $\lambda(t)$. 
$Q(t)$ can be expressed by the Riesz representation 
of the projector as 
\begin{equation}\label{eqn:RieszQ}
    Q(t) = -\frac{1}{2\pi \I}\int_{\Gamma(t)} R(z,t)dz
\end{equation}
in which $R(z,t) = (H(t)-z)^{-1}$ is the resolvent at time $t$ and the
complex contour can be chosen as
$\Gamma(t) = \{z \in \mathbb{C}: |z - \lambda(t)| = \Delta/2\}$.
Note that the assumption 2 assures that such representation is well-defined 
and, together with assumption 1, $Q(t)$ is actually also a smooth bounded map, 
which is summarized in the following lemma. 

\begin{lem}\label{LemBoundQ}
  The norms of all time derivatives $\norm{Q^{(k)}(t)}$ are bounded
  independently of $\epsilon$. 
\end{lem}
\begin{proof}
    We follow the technique in \cite{Teufel2003}. 
    The boundedness of $Q(t)$ directly follows 
    from the Riesz representation~\eqref{eqn:RieszQ}
    and the boundedness of $R(z,t)$ over the contour $\Gamma(t)$. 
    The contour $\Gamma(t)$ depends on $t$. To avoid taking time derivatives
    over the contour, note that the continuity of $\lambda(t)$ implies
    that for any $s \in [0,T]$, there exists a neighborhood $B(s,\delta_{s})$
    such that
    $$|z - \lambda(t)| \geq \Delta/4, \quad \forall t \in B(s, \delta_{s})\cap[0,T], \ z \in \Gamma(s) {.}$$
    By finding a finite cover $\bigcup_{j=1}^{n}B(s_j,\delta_{s_j}) \supset
    [0,T]$, for each $t \in [0,T]$, there exists a $s_j$
    such that $t \in B(s_j, \delta_{s_j})$ and we can rewrite $Q(t)$ as
    \begin{equation}
        Q(t) = -\frac{1}{2\pi \I}\int_{\Gamma(s_j)} R(z,t)dz {.}
    \end{equation}
    Such $s_j$ remains unchanged locally, hence
    \begin{align*}
      Q^{(k)}(t) = -\frac{1}{2\pi \I}\int_{\Gamma(s_j)} R^{(k)}(z,t)dz{.}
    \end{align*}
    The boundedness of $Q^{(k)}(t)$ can be directly assured by the boundedness
    of $H^{(k)}(t)$.
\end{proof}

\subsection{Adiabatic theorem}

First let us define the adiabatic evolution $\varphi_A(t)$ as
the solution to the following initial value problem 
\begin{equation}\label{eqn:AdiabaticWave2}
    \I \epsilon\partial_t \varphi_A = \I\epsilon[\dot{Q}, Q] \varphi_A,
    \quad \varphi_A(0) = \varphi(0) {.}
\end{equation}
Since the matrix $\I\epsilon [\dot{Q}, Q]$ is Hermitian,
$\norm{\varphi_A}_{2} = 1$ holds for all $t$.
Following the same proof of Eq.~\eqref{eqn:Pphi1} in Proposition
\ref{prop:ptrelation}, we find that 
$\varphi_A$ is an eigenvector of $H(t)$ corresponding 
to $\lambda(t)$, \ie 
$Q(t)\varphi_A(t) = \varphi_A(t)$ holds for all $t \in [0,T]$.


In the near adiabatic regime, we may separate $\varphi(t)$ into the
smooth component $\varphi_{A}$ and a remainder term. This is called 
the adiabatic theorem and is given in Theorem~\ref{thm:adiabatic1}.
\begin{thm}\label{thm:adiabatic1}
    Let $\varphi(t)$ follow the PT dynamics~\eqref{eqn:PTLSEsingle}, and 
    let $\varphi_A(t)$ follow the adiabatic evolution 
    as defined in Eq.~\eqref{eqn:AdiabaticWave2}. 
    Then the following decomposition
    \begin{equation}
        \varphi(t) = \varphi_A(t) + \epsilon \varphi_R(t)
    \end{equation}
    holds up to time $T = \Or(1)$. Furthermore,  $\varphi_R(t)$ is
    infinitely differentiable, and $\|\varphi_R(t)\|_2$ is bounded independently of $\epsilon$. 
\end{thm}
\begin{proof}
  The proof is organized according to the following three steps. 
  \begin{enumerate}
    \item Define another adiabatic evolution $\varphi_B$, 
      which satisfies an equation that resembles the PT dynamics. 
    \item Prove the adiabatic decomposition with respect to $\varphi_B$, 
      \ie there exists an infinitely differentiable function $\eta(t)$ such that
      $$\varphi(t) = \varphi_B(t) + \epsilon \eta(t), \quad \forall t\in [0,T] {,}$$
      where $\|\eta(t)\|_2$ is bounded independently of $\epsilon$. 
    \item Prove that the difference between $\varphi_B$ and $\varphi_A$ 
      is of $\Or(\epsilon)$. 
  \end{enumerate}

  1. Define $\mathcal{T}_{B}$ as the solution to the initial value
  problem
  \begin{equation}\label{eqn:AdiabaticEvolution}
    \I \epsilon \partial_t \mathcal{T}_{B}
    = (H - \varphi^*H\varphi + \I \epsilon[\dot{Q},Q])\mathcal{T}_{B}, \quad
    \mathcal{T}_{B}(0) = I{,}
  \end{equation}
  We define $\varphi_B$ according to
  $$\varphi_B(t) := \mathcal{T}_{B}(t)\varphi(0){,}$$
  which solves the initial value problem
  \begin{equation}\label{eqn:AdiabaticWave}
    \I \epsilon \partial_t \varphi_B
    = (H - \varphi^*H\varphi + \I \epsilon[\dot{Q},Q])\varphi_B, \ \ \ \
    \varphi_B(0) = \varphi(0){.}
  \end{equation}
  Since the matrix $(H - \varphi^*H\varphi + \I\epsilon[\dot{Q},Q])$ is Hermitian, 
  $\mathcal{T}_{B}$ is a unitary evolution,
  and $\varphi_B$ is a normalized vector.

  Next we show that $\varphi_B(t)$ is an eigenvector of $H(t)$ 
  corresponding to $\lambda(t)$, \ie
  \begin{equation}\label{eqn:AdiabaticvarphiB}
    Q(t)\varphi_B(t) = \varphi_B(t) {.}
  \end{equation}
  This can be done by showing that $Q\varphi_B$ and $\varphi_B$ solve
  the same initial value problem. 
  By the Leibniz rule and Eq.~\eqref{eqn:AdiabaticWave}, we have
  \begin{align*}
    \partial_t{(Q\varphi_B)} &= \dot{Q}\varphi_B + Q\dot{\varphi_B} \\
    &= \dot{Q}\varphi_B + Q[\dot{Q},Q]\varphi_B - \frac{\I}{\epsilon}Q(H - \varphi^*H\varphi)\varphi_B{.}
  \end{align*}
  Use the identities similar to~\eqref{eqn:Pidentity}, 
  $$\dot{Q} = \dot{Q}Q+Q\dot{Q}, \quad Q\dot{Q}Q = 0, \quad Q^2 = Q {,}$$
  we have
  \begin{align*}
    \dot{Q} + Q[\dot{Q},Q]
    &= \dot{Q}Q+Q\dot{Q} + Q\dot{Q}Q - Q^2\dot{Q} \\
    &= \dot{Q}Q
    = (\dot{Q}Q-Q\dot{Q})Q
    = [\dot{Q},Q]Q {.}
  \end{align*}
  Hence
  \begin{align*}
    \partial_t{(Q\varphi_B)} = [\dot{Q},Q]Q\varphi_B - \frac{\I}{\epsilon}Q(H - \varphi^*H\varphi)\varphi_B{.}
  \end{align*}
  Together with the identity $QH=HQ$, we have
  \begin{align*}
    \partial_t{(Q\varphi_B)} &= [\dot{Q},Q]Q\varphi_B - \frac{\I}{\epsilon}(H - \varphi^*H\varphi)Q\varphi_B \\
    &= -\frac{\I}{\epsilon}(H - \varphi^*H\varphi + i\epsilon[\dot{Q},Q])(Q\varphi_B){.}
  \end{align*}
  Furthermore, the initial condition 
  satisfies $Q(0)\varphi_{B}(0)=\varphi_{B}(0)=\varphi(0)$.
  Hence $Q\varphi_B$ solves the same initial value problem
  \eqref{eqn:AdiabaticWave} as $\varphi_B$. 

  In summary, in step 1 we define another adiabatic evolution $\varphi_B(t)$ 
  which is also an eigenstate of $H(t)$ corresponding to $\lambda(t)$ 
  (Eq.~\eqref{eqn:AdiabaticvarphiB}). 
  Therefore, $\varphi_A(t)$ and $\varphi_B(t)$ are both eigenstates of
  $H(t)$ differing at most by a choice of gauge.


  2. Now we estimate the distance between $\varphi(t)$ and $\varphi_B(t)$. 
  This can be done by mimicking the standard proof of the adiabatic theorem \cite{AvronElgart1999} with some modifications. 
  By the definition of $\varphi_B$,
  $$\|\varphi(t)-\varphi_B(t)\|_2 = \|\varphi(t)-\mathcal{T}_{B}(t)\varphi(0)\|_2
  = \|\mathcal{T}_{B}^{-1}(t)\varphi(t)-\varphi(0)\|_2 {.}$$
  Define $w(t) = \mathcal{T}_{B}^{-1}(t)\varphi(t)$, then 
  \begin{equation}\label{eqn:AdiabaticProofest}
    \|\varphi(t)-\varphi_B(t)\|_2 = \|w(t)-w(0)\|_2 = \left\|\int_0^t
    \dot{w}(s) ds\right\|_2 {.}
  \end{equation}
  In order to estimate $\dot{w}$,
  differentiate the equation $\mathcal{T}_{B}w = \varphi$ and we get
  \begin{equation}\label{eqn:defDw}
    \dot{w} = - \mathcal{T}_{B}^{-1}[\dot{Q},Q]\mathcal{T}_{B}w {.}
  \end{equation}
  Note that if we define
  $$X(t) = -\frac{1}{2\pi \I} \int_{\Gamma(s_j)} R(z,t)\dot{Q}(t)R(z,t)dz$$
  where $\Gamma(s_j)$ and $R(z,t)$ are defined in the proof of Lemma \ref{LemBoundQ},
  then $\|X\|_2$ and $\|\dot{X}\|_2$ are bounded independently of $\epsilon$, 
  and~\cite{AvronElgart1999,Teufel2003}
  $$[\dot{Q},Q] = [H, X] {.}$$
  Then 
  \begin{equation}\label{eqn:AdiabaticProofCom1}
    \dot{w} = - \mathcal{T}_{B}^{-1}[H,X]\mathcal{T}_{B}w 
    = -(\mathcal{T}_{B}^{-1}H)X\mathcal{T}_{B}w + \mathcal{T}_{B}^{-1}X(H\mathcal{T}_{B})w  {.}
  \end{equation}
  To compute the first part of Eq.~\eqref{eqn:AdiabaticProofCom1}, 
  we first take the time derivative of the identity $I = \mathcal{T}_{B}^{-1}\mathcal{T}_{B}$ and get
  \begin{equation}
    \mathcal{T}_{B}^{-1}H = -\I\epsilon\partial_t{(\mathcal{T}_{B}^{-1})} + (\varphi^*H\varphi)\mathcal{T}_{B}^{-1} - \I\epsilon \mathcal{T}_{B}^{-1}[\dot{Q},Q] {.}
  \end{equation}
  Then the first part of Eq.~\eqref{eqn:AdiabaticProofCom1} can be rewritten as
  \begin{equation}\label{eqn:AdiabaticProofCom2}
    \begin{split}
      -(\mathcal{T}_{B}^{-1}H)X\mathcal{T}_{B}w 
      &= \I\epsilon\partial_t{(\mathcal{T}_{B}^{-1})}X\mathcal{T}_{B}w + \I\epsilon \mathcal{T}_{B}^{-1}[\dot{Q},Q]X\mathcal{T}_{B}w - (\varphi^*H\varphi)\mathcal{T}_{B}^{-1}X\mathcal{T}_{B}w {.}
    \end{split}
  \end{equation}
  To compute the second part of Eq.~\eqref{eqn:AdiabaticProofCom1}, 
  rewrite Eq.~\eqref{eqn:AdiabaticEvolution} as 
  \begin{equation}
    H\mathcal{T}_{B} = \I\epsilon \dot{\mathcal{T}_{B}} 
    + (\varphi^*H\varphi)\mathcal{T}_{B} - \I\epsilon[\dot{Q},Q]\mathcal{T}_{B}{,}
  \end{equation}
  and then
  \begin{equation}\label{eqn:AdiabaticProofCom3}
    \mathcal{T}_{B}^{-1}X(H\mathcal{T}_{B})w 
    = \I\epsilon\mathcal{T}_{B}^{-1}X\dot{\mathcal{T}_{B}}w -\I\epsilon\mathcal{T}_{B}^{-1}X[\dot{Q},Q]\mathcal{T}_{B}w +(\varphi^*H\varphi)\mathcal{T}_{B}^{-1}X\mathcal{T}_{B}w {.}
  \end{equation}
  Sum up Eq.~\eqref{eqn:AdiabaticProofCom2} and~\eqref{eqn:AdiabaticProofCom3}, 
  then Eq.~\eqref{eqn:AdiabaticProofCom1} becomes
  \begin{equation}\label{eqn:AdiabaticDw}
    \dot{w} = \I\epsilon(\partial_t{(\mathcal{T}_{B}^{-1})}X\mathcal{T}_{B}+\mathcal{T}_{B}^{-1}X\dot{\mathcal{T}_{B}})w + \I\epsilon \mathcal{T}_{B}^{-1}[[\dot{Q},Q],X]\mathcal{T}_{B}w {.}
  \end{equation}
  In Eq.~\eqref{eqn:AdiabaticDw}, 
  the second term of the right hand side is already of $\Or(\epsilon)$. 
  Now we turn to the first term to treat the derivatives
  $\partial_t{(\mathcal{T}_{B}^{-1})}$ and $\dot{\mathcal{T}_{B}}$.
  By repeated usage of the Leibniz rule, Eq.~\eqref{eqn:AdiabaticDw} becomes
  \begin{equation}\label{eqn:AdiabaticDw2}
    \begin{split}
      \dot{w} &= \I\epsilon\partial_t{(\mathcal{T}_{B}^{-1}X\mathcal{T}_{B})}w 
      - \I\epsilon\mathcal{T}_{B}^{-1}\dot{X}\mathcal{T}_{B}w
      + \I\epsilon \mathcal{T}_{B}^{-1}[[\dot{Q},Q],X]\mathcal{T}_{B}w\\
      &= \I\epsilon\partial_t{(\mathcal{T}_{B}^{-1}X\mathcal{T}_{B}w)}
      - \I\epsilon\mathcal{T}_{B}^{-1}X\mathcal{T}_{B}\dot{w}
      - \I\epsilon\mathcal{T}_{B}^{-1}\dot{X}\mathcal{T}_{B}w
      + \I\epsilon \mathcal{T}_{B}^{-1}[[\dot{Q},Q],X]\mathcal{T}_{B}w\\
      &= \I\epsilon\partial_t{(\mathcal{T}_{B}^{-1}X\varphi)}
      + \I\epsilon\mathcal{T}_{B}^{-1}X[H,X]\varphi
      - \I\epsilon\mathcal{T}_{B}^{-1}\dot{X}\varphi
      + \I\epsilon \mathcal{T}_{B}^{-1}[[\dot{Q},Q],X]\varphi{.}
    \end{split}
  \end{equation}
  In the last equation we use again Eq.~\eqref{eqn:defDw}. 
  Substitute Eq.~\eqref{eqn:AdiabaticDw2} back to Eq.~\eqref{eqn:AdiabaticProofest}, 
  we get
  \begin{equation}\label{eqn:AdiaProof2}
    \begin{split}
      \|\varphi(t)-\varphi_B(t)\|_2 &= \|\int_0^t \dot{w}(s) ds\|_2\\
      &\leq \epsilon \|(\mathcal{T}_{B}^{-1}X\varphi)(t) - (\mathcal{T}_{B}^{-1}X\varphi)(0)\|_2 \\
      &\quad + \epsilon\left\|\int_{0}^t \big(
      \mathcal{T}_{B}^{-1}X[H,X]\varphi -
      \mathcal{T}_{B}^{-1}\dot{X}\varphi +
      \mathcal{T}_{B}^{-1}\big[[\dot{Q},Q],X\big]\varphi\big) ds \right\|_2\\
      &= \Or(\epsilon) {.}
    \end{split}
  \end{equation}
  Therefore there exists $\eta(t)$ such that
  \begin{equation}\label{eqn:AdiabaticIm}
    \varphi(t) = \varphi_B(t) + \epsilon \eta(t) {,}
  \end{equation}
  where $\|\eta(t)\|_2$ is bounded independently of
  $\epsilon$.  The differentiability of $\eta(t)$ follows
  directly from that of $\varphi(t)$ and $\varphi_{B}(t)$.

  3. Comparing Eq. \eqref{eqn:AdiabaticIm} with our goal, 
  the only thing that we need to prove is that the distance between
  $\varphi_B$ and $\varphi_A$ is also $\Or(\epsilon)$. 
Note that $\varphi_{A}$ can be written as~\cite{FetterWalecka2003}
  \begin{equation}
    \varphi_{A}(t) = \mathfrak{T}\left[\exp\left( \int_{0}^{t}
    [\dot{Q}(s),Q(s)] ds  \right)\right] \varphi_{A}(0),
    \label{}
  \end{equation}
  where $\mathfrak{T}$ is the time ordering operator due to the explicit
  time dependence of $Q$. 
  Using the power series representation, the time-ordered exponential
  is defined as
  \begin{equation}
    \mathfrak{T} \Bigl[ e^{\int_0^t A(s) \ud s} \Bigr] = I + \int_0^t A(s) \ud s
    + \frac{1}{2!}  \int_0^t \int_0^t \mathfrak{T} [A(s_1) A(s_2)] \ud s_1 \ud s_2 + \cdots,
  \end{equation}
  where the time-ordered product of two matrices $\mathfrak{T}[A(s_1) A(s_2)]$
  is given by
  \begin{equation}
    \mathfrak{T}[A(s_1) A(s_2)] = 
    \begin{cases}
      A(s_1) A(s_2), & s_1 \ge s_2; \\
      A(s_2) A(s_1), & s_1 < s_2.
    \end{cases}
  \end{equation}

  Using Duhamel's principle, we have from Eq. \eqref{eqn:AdiabaticWave2} and
  \eqref{eqn:AdiabaticWave}
  \begin{equation}\label{eqn:AGlemma}
    \varphi_B(t) = \varphi_A(t)
    + \int_0^t
    \mathfrak{T}\left[\exp\left( \int_{s}^{t}
    [\dot{Q}(s'),Q(s')]ds'\right)\right]
    \cdot \frac{1}{\I\epsilon}(H(s)-\varphi^*(s)H(s)\varphi(s))\varphi_B(s)ds
  \end{equation}
  By Eq.~\eqref{eqn:AdiabaticvarphiB}, \eqref{eqn:AdiabaticIm}, and the
  normalization condition of $\varphi$ and $\varphi_B$, 
  \begin{equation}
    \begin{split}
      (H - \varphi^* H \varphi)\varphi_B
      &= -\lambda(\epsilon\eta^*\varphi_B+\epsilon\varphi_B^*\eta)\varphi_B - \epsilon^2(\eta^*H\eta)\varphi_B \\
      &= -\lambda[(\varphi_B+\epsilon\eta)^*(\varphi_B+\epsilon\eta)-\varphi_B^*\varphi_B - \epsilon^2\eta^*\eta]\varphi_B 
      - \epsilon^2(\eta^*H\eta)\varphi_B \\
      &= \epsilon^2\lambda(\eta^*\eta)\varphi_B - \epsilon^2(\eta^*H\eta)\varphi_B \\
      &= \Or(\epsilon^2) {.}
    \end{split}
  \end{equation}
  Hence Eq.~\eqref{eqn:AGlemma} implies 
  \begin{equation}
    \varphi_B - \varphi_A = \Or(\epsilon) {.}
  \end{equation}
  Therefore, $\varphi_R := \eta + (\varphi_B-\varphi_A)/\epsilon$ is
  infinitely differentiable,
  and $\|\varphi_R(t)\|_2$ is bounded independently of $\epsilon$. This
  proves the decomposition of the solution to the PT dynamics
  \begin{equation}
    \varphi = \varphi_B + \epsilon\eta 
    = \varphi_B + \epsilon\varphi_R - (\varphi_B-\varphi_A)
    =\varphi_A + \epsilon \varphi_R {.}
  \end{equation}
\end{proof}

Theorem~\ref{thm:adiabatic1} gives a decomposition near the adiabatic regime 
with respect to the PT wave function. 
As a corollary, we also have the adiabatic theorem with respect to the
projector. 
\begin{cor}\label{cor:adiabaticP}
    For the projector $P(t)$, there exists an infinitely differentiable 
    matrix-valued function $R(t)$ such that 
    \begin{equation}
       P(t) = Q(t) + \epsilon R(t)
    \end{equation}
    holds for all $t$ up to $T = \Or(1)$, 
    where $\|R(t)\|_2$ is bounded independently of $\epsilon$. 
\end{cor}
\begin{proof}
    This follows directly from theorem~\ref{thm:adiabatic1}
    \begin{equation}
        \begin{split}
            P &= \varphi \varphi^*
            = (\varphi_A+\epsilon \varphi_R)(\varphi_A+\epsilon \varphi_R)^*\\
            &= Q + \epsilon(\varphi_R\varphi_A^* + \varphi_A\varphi_R^* + \epsilon\varphi_R\varphi_R^*){.}
        \end{split}
    \end{equation}
\end{proof}

\begin{rem}
The adiabatic theorem for the Schr\"odinger wave function $\psi(t)$ 
has been well established in the
literature \eg~\cite{Kato1950,AvronElgart1999,Teufel2003}, 
where the decomposition takes the form
$\psi = \psi_A + \epsilon \tilde{\psi}_R$, 
and the adiabatic
evolution $\psi_A$ satisfies 
\begin{equation}\label{eqn:PTeigen_def_standand}
    \I \epsilon \partial_t\psi_A = (H+\I\epsilon[\dot{Q},Q])\psi_A.
\end{equation}
\REV{We compare our result with previous well-established ones from two aspects. First, there is an important difference between the PT eigenfunction $\varphi_A$, governed by Eq.~\eqref{eqn:AdiabaticWave2}, and the standard one $\psi_A$, governed by Eq.~\eqref{eqn:PTeigen_def_standand}. Although both $\varphi_A$ and $\psi_A$ are  eigenfunctions of $H(t)$, their phase factors are different, resulting in different oscillatory behavior. More specifically, the standard  wavefunction $\psi_A$ oscillates on the scale of $\Or(\epsilon^{-1})$ since (at least intuitively) Eq.~\eqref{eqn:PTeigen_def_standand} is just a small perturbation of the original Schr\"odinger dynamics. The PT eigenfunction $\varphi_A$ does not 
depend on $\epsilon$, and thus oscillates on the scale of $\Or(1)$. 
When projected to the eigenspace, the PT dynamics leads to the optimal phase factor, and this verifies the effectiveness of the definition of PT (to minimize unnecessary oscillations) and provides another theoretical explanation of the performance shown in Fig.~\ref{fig:example_NLS_wave:a}. 
Second, our proof largely follows the existing works of the adiabatic theorem~\cite{Kato1950,AvronElgart1999,Teufel2003}. Our main modification is to address the special non-linear term in the PT dynamics, even though the original Schr\"odinger dynamics is linear.}

\end{rem}


\begin{rem}
As mentioned at the end of step 1, 
$\varphi_B$ is also an eigenstate,
and Eq.~\eqref{eqn:AdiabaticIm} indeed leads to another version of
the adiabatic theorem, but with notable differences from the
decomposition in Theorem~\ref{thm:adiabatic1}.
First, the definition of $\varphi_B$
still relies on the information of $\varphi$, and thus is not a
self-contained equation. 
Second, the norms of the derivatives of $\varphi_B$ still depend on
$\epsilon$ (more precisely one can prove that
$\norm{\varphi_B^{(k)}}_{2} \sim \Or(1/\epsilon^{k-2})$ for $k \geq 3$), which indicates that the gauge
choice of $\varphi_B$ is not optimal either.
\end{rem}

\subsection{Local truncation error}


In this section, we show that after time discretization, the local
truncation error of the discretized PT dynamics improves by one order in
terms of $\epsilon$ compared to that of the discretized Schr\"odinger
dynamics in the near adiabatic regime.  This is given in
Lemma~\ref{lem:LTE}.

\REV{For simplicity we will focus on the numerical integrators in the classes of  Runge-Kutta methods and linear multistep methods, both of which are widely used for simulating the Schr\"odinger equation. We will refer numerical integrator to either a Runge-Kutta method or a linear multistep method in our context.}
Recall that a numerical integrator with a given time step $h$, denoted by $I_{h}$, can be generally written as
\begin{equation}
  u_{n+1} = I_h(u_{n},\cdots,u_{n-l}),
\end{equation}
for some integer $l\ge 0$, and $u_{n}$ is the numerical approximation to
the true solution $u(t_{n})$.
If $I_h$ is of order $k$, then the local truncation error at step $n+1$,
defined as 
\[
L_{n+1} = I_h(u(t_n),\cdots,u(t_{n-l})) - u_{n+1},
\]
should satisfy
\[
\norm{L_{n+1}}_{2} \le  Ch^{k+1} \norm{u^{(k+1)}(\xi_{n+1})}_{2}, 
\]
for some $\xi_{n+1}\in [t_{n},t_{n+1}]$.
When applied to the Schr\"odinger dynamics, the PT dynamics, or
the associated Hamiltonian form, we may identify $u$ with $\psi$, 
$\varphi$, or the equivalent $(q,p)$ representation.

\begin{lem}\label{lem:LTE}
  Apply a numerical integrator of order $k$ 
  to the Schr\"odinger dynamics or its Hamiltonian
  form~\eqref{eqn:LSEHamiltonian}. Then the local truncation error is
  bounded by $C h^{k+1}/\epsilon^r$ up to the time $T\sim \Or(1)$,
  with $r = k+1$ and $C$ is a constant independent of $h$ and $\epsilon$.  The same result holds for the PT dynamics \eqref{eqn:PTLSEsingle} 
  or its corresponding Hamiltonian form \eqref{eqn:PTHamiltonianFor} with
  $r=k$.
\end{lem}

\begin{proof}
   It is sufficient to show that the derivatives satisfy
  $\norm{\psi^{(k+1)}}_{2}\leq \Or(1/\epsilon^{k+1})$, and
  $\norm{\varphi^{(k+1)}}_{2} \leq \Or(1/\epsilon^k)$ for any $k\ge 0$.
  This can be proved by induction.
 

  1. For $\psi$, 
  the case $k = 0$ directly follows from Eq.~\eqref{eqn:SEgeneral}. 
  Assume the estimate holds for all the integers smaller than $k$,
  differentiate the Schr\"odinger equation $k$ times and we get
  \begin{equation}\label{pf:d2}
    \psi^{(k+1)} = \frac{1}{\I \epsilon}
    \sum_{j=0}^k\binom{k}{j} H^{(k-j)}\psi^{(j)} {.}
  \end{equation}
   By the induction and the assumption 1, 
  \begin{equation}
    \|\psi^{(k+1)}\|_2 \leq \frac{C}{\epsilon}\sum_{j=0}^k\binom{k}{j} \frac{1}{\epsilon^j}
    \sim \Or(\epsilon^{-(k+1)}).
  \end{equation}

  2. For $\varphi$,
  we first study the derivatives of $P$,
  and then use the PT condition~\eqref{eqn:PTcondition} 
  to obtain the derivatives of $\varphi$.

  By Corollary~\ref{cor:adiabaticP}, the von Neumann
  equation~\eqref{eqn:vonNeumann} and the identity $HQ = QH$,
  the first order derivative of $P$ satisfies
  $$\|\dot{P}\|_2 = \frac{1}{\epsilon}\|HP-PH\|_2 = \|HR-RH\|_2 \leq
  \Or(1){.}$$
  Furthermore, differentiate the von Neumann
  equation~\eqref{eqn:vonNeumann} $k$ times, we get
  \begin{equation}
    P^{(k+1)} = \frac{1}{\I\epsilon}\sum_{j=0}^k\binom{k}{j}[H^{(j)},P^{(k-j)}] {,}
  \end{equation}
  from which we can show by induction that
  \begin{equation}\label{eqn:orderP}
    \norm{P^{(k+1)}}_{2} \leq \Or(\epsilon^{-k}) {.}
  \end{equation}

  Now use the PT condition $P\dot{\varphi} = 0$, we find for $k=0$,
  \begin{equation}
    \dot{\varphi} = \partial_{t}(P\varphi) = \dot{P}\varphi \leq \Or(1) {.}
  \end{equation}
  Furthermore,
  \begin{equation}
    \varphi^{(k+1)} = \sum_{j=0}^k\binom{k}{j} P^{(j+1)}\varphi^{(k-j)} {,}
  \end{equation}
  from which we can prove by induction and Eq.~\eqref{eqn:orderP} that
  \begin{equation}
    \varphi^{(k+1)} \leq \Or(\epsilon^{-k}) {.}
  \end{equation}
\end{proof}


\subsection{Global error}


The analysis of the local truncation error directly extends to the
global error up to $T\sim \Or(\epsilon)$, following the classical stability
analysis.  However, the
Lipschitz constants corresponding to the right hand side of the
Schr\"odinger dynamics and the PT dynamics are generally
$\Or(1/\epsilon)$, which leads to an exponentially growing factor
$\exp(T/\epsilon)$ in the global error bounds. Hence we cannot directly
obtain the global error estimate up to $\Or(1)$ time.

However, if we adopt the Hamiltonian formulation of the dynamics and
employ a symplectic integrator, we can indeed obtain long time error
estimates. This is stated in Theorem~\ref{thm:GE}, of which the proof 
directly follows from Lemma~\ref{lem:LTE} and Theorem X.3.1 in \cite{HairerLubichWanner2006}.
\begin{thm}\label{thm:GE}
    Apply a symplectic integrator of order $k$ to the
    Hamiltonian system \eqref{eqn:LSEHamiltonian} and
    \eqref{eqn:PTHamiltonianFor}, then there exist constants $c,C$, 
    independent of $h$ and $\epsilon$, such
    that for the time step $h \leq c\epsilon$, 
    the numerical solutions up to the time $T \sim \Or(1)$ satisfy
    \begin{equation}
      \|(q_n,p_n) - (q(t),p(t))\|_{2} \leq C\frac{h^k}{\epsilon^r}.
    \end{equation}
    Here $r = k + 1$ for the Schr\"odinger dynamics \eqref{eqn:LSEHamiltonian}
    and $r = k$ for the PT dynamics \eqref{eqn:PTHamiltonianFor}.
\end{thm}

\begin{rem}
In Theorem X.3.1 in \cite{HairerLubichWanner2006}, 
all terms are bounded by $\Or(1)$
terms and there is no $\epsilon$ dependence.
In order to adapt its proof  to the current situation,
we observe the key fact in Theorem X.3.1 in \cite{HairerLubichWanner2006} that the global error of a symplectic integrator
accumulates linearly in time with no exponential growing factor. Therefore the local truncation error which is $\Or(h^{k+1}/\epsilon^r)$ 
directly sums up linearly to the global error of 
$\Or(h^{k}/\epsilon^r)$. 
\end{rem}

\begin{rem}
The nontrivial restriction on the time step size $h \leq c\epsilon$ is
because Theorem X.3.1 in \cite{HairerLubichWanner2006} holds only for
sufficiently small time steps.  In general, $h$ must be no larger
than $c/L$ where $L$ is the Lipschitz constant of the right hand side of
the Hamiltonian system, and is $\Or(1/\epsilon)$ in the singularly
perturbed regime. Nonetheless, numerical results in
Section~\ref{sec:numer} indicate that the PT dynamics may admit a
considerably larger time step in practice.  
\end{rem}

\begin{rem}
When a symplectic integrator is used, Theorem~\ref{thm:GE} is
directly applicable to the Schr\"odinger dynamics. However, the PT
dynamics~\eqref{eqn:PTLSEsingle} and the Hamiltonian
system~\eqref{eqn:PTHamiltonianFor} share the same exact solution, but
lead to different numerical schemes even when the same integrator is
used. Despite such difference, numerical results 
in Section~\ref{sec:numer} indicate that the
symplectic integrators, and even certain non-symplectic schemes, can still
perform very well in the PT dynamics~\eqref{eqn:PTLSEsingle}. 
\end{rem}


\begin{rem}
  Theorem~\ref{thm:GE} also indicates that the PT dynamics is relatively more effective
  when combined with low order methods.  For instance, if we
  would like to achieve some desired accuracy $\delta$ 
  (assuming $\delta$ is sufficiently small), 
  then for the Schr\"odinger dynamics, we should 
  choose the time step size to be 
  $$h \sim \Or(\delta^{\frac{1}{k}}\epsilon^{1+\frac{1}{k}}).$$ 
  For the PT dynamics, we should choose 
  $$h \sim \Or(\delta^{\frac{1}{k}}\epsilon).$$ 
  From this perspective, the gain of the PT dynamics is less significant
  when $k$ is large.
\end{rem}

\section{Numerical results}\label{sec:numer}

In this section we study the effectiveness of the PT dynamics 
using three examples. 
The first one is a toy example, which is a linear Schr\"odinger equation
in $\CC^{2}$.
This example gives a clear illustration of the performance
of different numerical methods near and beyond the adiabatic regime. 
The second example is a nonlinear Schr\"odinger equation in a
one-dimensional space, where we also compare the computational cost
between the propagation of the Schr\"odinger dynamics and the PT
dynamics.  In the end we study the electron dynamics of a benzene
molecule driven by an ultrashort laser pulse described by the
time-dependent density functional theory (TDDFT).


The test programs in the first two examples are written in MATLAB. We
implement the PT dynamics for TDDFT in the PWDFT code, which performs
planewave based electronic structure calculations. PWDFT is a
self-contained module in the massively parallel DGDFT (Discontinuous
Galerkin Density Functional Theory) software package written in
MPI and C++~\cite{LinLuYingE2012,HuLinYang2015a}. All calculations are
carried out using the BRC High Performance Computing service. Each
node consists of two Intel Xeon 10-core Ivy Bridge processors (20 cores
per node) and 64 gigabyte
(GB) of memory. We use the Anderson
mixing for solving all the nonlinear fixed point problems, 
including those in the PT dynamics, the nonlinear Schr\"odinger equation, 
and the TDDFT calculations. 
Here no preconditioner is used for the first two examples. We use a shifted
Laplace preconditioner for the TDDFT example, which can be implemented
efficiently in the planewave basis set using the fast Fourier transform. 

%

%

\subsection{A toy example}

First we present a linear example, 
in which $H(t)$ is chosen to be
\begin{equation}
    H(t) = \left(
                 \begin{array}{cc}
                   t-t_0 & \delta \\
                   \delta & -(t-t_0) \\
                 \end{array}
               \right) {.}
\end{equation}
Here $H(t)$ has the eigenvalues $\lambda_{1,2}(t) = \mp \sqrt{(t-t_0)^2+\delta^2}$, 
where $\delta>0$ ensures the gap condition and
controls the size of the gap. 
When $\delta$ is large, the dynamics stays closer to
the adiabatic regime, while the dynamics can go beyond the adiabatic
regime with a smaller $\delta$  
(see Fig.~\ref{fig:example_2by2_gap}). 
The initial value is always chosen to be 
the normalized eigenvector of $H(0)$ corresponding to 
$\lambda_1(0) = -\sqrt{t_0^2+\delta^2}$. 
We propagate the wave functions up to $T = 1$.
For the choices of the parameters in the Anderson Mixing 
in propagating PT dynamics, 
the step length $\alpha = 1$, the mixing dimension is 20, 
and the tolerance is $10^{-8}$. 

\begin{figure} \centering
    \begin{subfigure}[b]{0.48\textwidth}
        \includegraphics[width=\textwidth]{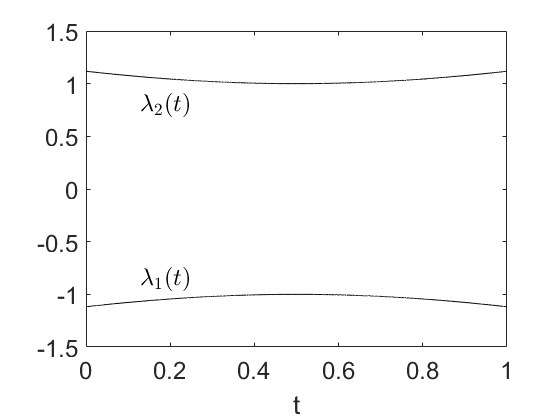}
        \caption{$\delta = 1$}
        \label{fig:example_2by2_gap:a}
    \end{subfigure} 
    \begin{subfigure}[b]{0.48\textwidth}    
        \includegraphics[width=\textwidth]{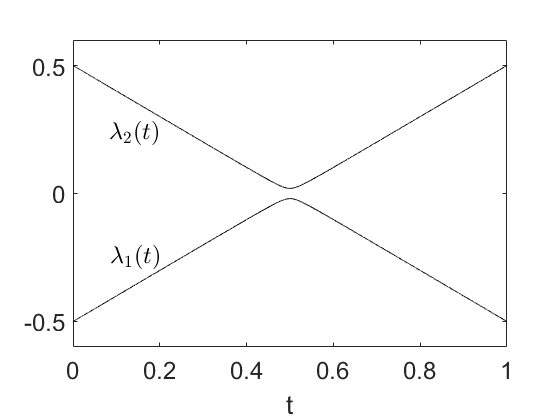}
        \caption{$\delta = 0.02$}
        \label{fig:example_2by2_gap:b}    
    \end{subfigure} 
    \caption{Eigenvalues of $H(t)$ in the toy example with $t_0 = 0.5$
    and two choices of $\delta$.
    }
    \label{fig:example_2by2_gap}
\end{figure}


\subsubsection{Near adiabatic regime}
First we consider the near adiabatic case with $\delta = 1$.  
We compare the following numerical methods: 
\begin{itemize}
  \item S-RK4: fourth order Runge-Kutta method (RK4) 
    applied to the Schr\"odinger equation~\eqref{eqn:SEgeneral}
  \item PT-RK4: fourth order Runge-Kutta method (RK4) 
  applied to the PT dynamics~\eqref{eqn:PTLSEsingle}
  \item S-GL2: implicit midpoint rule (GL2) applied to 
  the Schr\"odinger equation~\eqref{eqn:SEgeneral}
\item PT-Ham-GL2: implicit midpoint rule (GL2) applied to 
the PT Hamiltonian system~\eqref{eqn:PTHamiltonianFor}
  \item PT-GL2: implicit midpoint rule (GL2) applied to 
  the PT dynamics~\eqref{eqn:PTLSEsingle}
  \item PT-CN: trapezoidal rule (or the Crank-Nicolson method, CN) 
  applied to the PT dynamics~\eqref{eqn:PTLSEsingle}
\end{itemize}

\begin{figure} \centering
    \begin{subfigure}[b]{0.48\textwidth}
        \includegraphics[width=\textwidth]{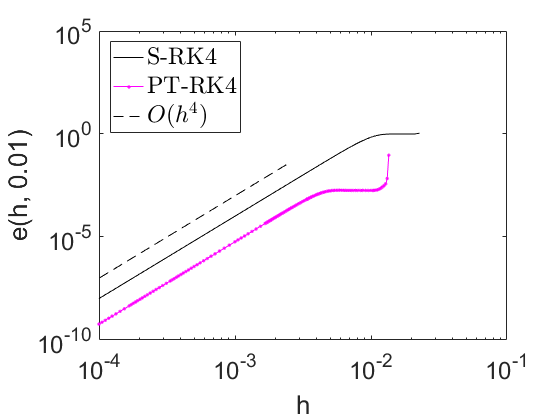}
        \caption{$\epsilon = 0.01$}
        \label{fig:example_2by2_errors:a}
    \end{subfigure} 
    \begin{subfigure}[b]{0.48\textwidth}    
        \includegraphics[width=\textwidth]{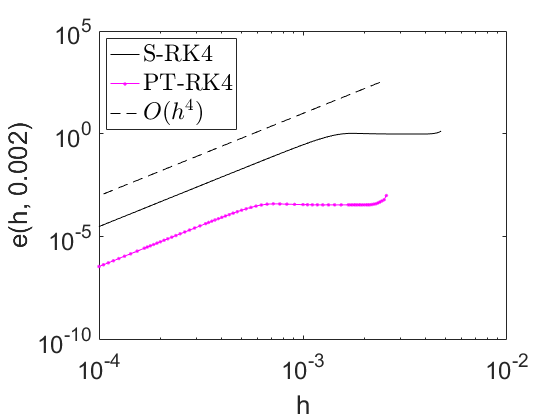}
        \caption{$\epsilon = 0.002$}
        \label{fig:example_2by2_errors:b}    
    \end{subfigure} 
    \begin{subfigure}[b]{0.48\textwidth}    
        \includegraphics[width=\textwidth]{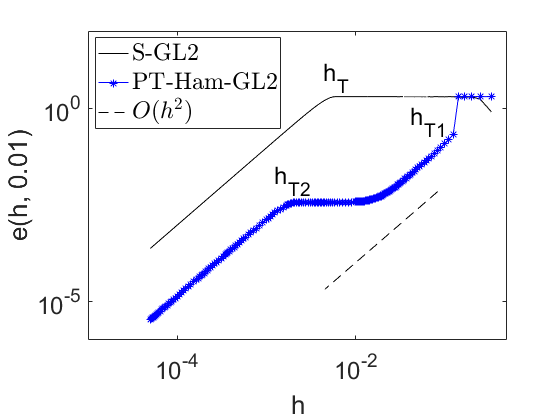}
        \caption{$\epsilon = 0.01$}
        \label{fig:example_2by2_errors:c}    
    \end{subfigure} 
    \begin{subfigure}[b]{0.48\textwidth}    
        \includegraphics[width=\textwidth]{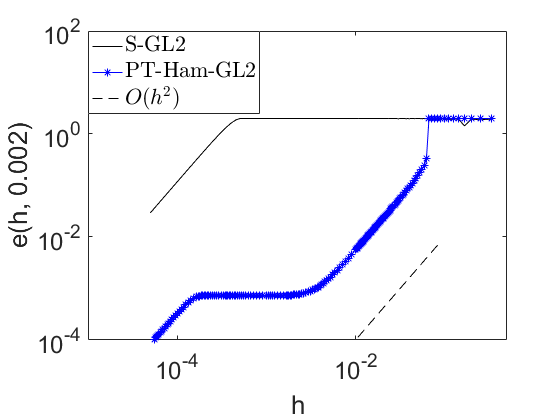}
        \caption{$\epsilon = 0.002$}
        \label{fig:example_2by2_errors:d}    
    \end{subfigure} 
    \caption{Numerical errors of different numerical methods 
    in the near adiabatic regime of the toy example. 
    (a)(b) compare S-RK4 and PT-RK4
    for $\epsilon = 0.01,0.002$, respectively. 
    (c)(d) compare S-GL2 and PT-Ham-GL2 
    for $\epsilon = 0.01,0.002$, respectively. 
    } 
    \label{fig:example_2by2_errors}
\end{figure}

Fig.~\ref{fig:example_2by2_errors} compares the performance
of different methods for this toy example.  
The numerical error is computed by  
$$\mathsf{e}(h,\epsilon) = \max_{n \text{ s.t. } nh \in [0,T]}\|u_n-u(t_n)\|_2$$
where $u$ denotes $\psi$ for the Schr\"odinger dynamics, 
$\varphi$ for the PT dynamics and $(q,p)$ for the Hamiltonian systems,
respectively. 

\begin{figure} \centering
    \begin{subfigure}[b]{0.48\textwidth}
        \includegraphics[width=\textwidth]{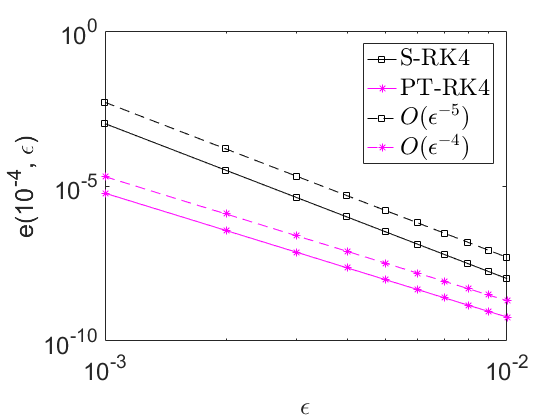}
        \caption{}
        \label{fig:example_2by2_order:a}
    \end{subfigure} 
    \begin{subfigure}[b]{0.48\textwidth}
        \includegraphics[width=\textwidth]{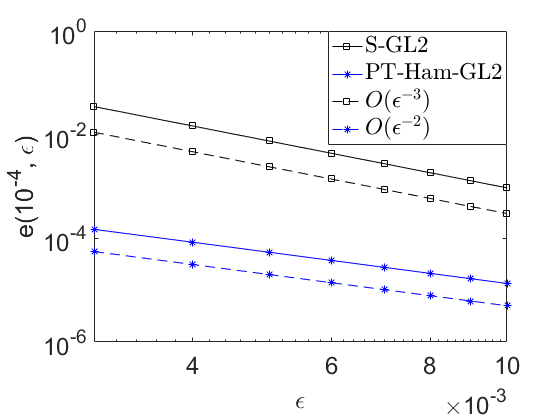}
        \caption{}
        \label{fig:example_2by2_order:b}
    \end{subfigure}
    \caption{Relationship between the asymptotic errors and $\epsilon$ 
    in the near adiabatic regime of the toy example. 
    Here we fix the time step size to be $h = 10^{-4}$ for both (a)(b).}
    \label{fig:example_2by2_order}
\end{figure}

We first consider the explicit numerical methods. 
Fig.~\ref{fig:example_2by2_errors:a} and~\ref{fig:example_2by2_errors:b} 
give a comparison between S-RK4 and PT-RK4. 
Not surprisingly, as an explicit method, RK4 is numerically unstable 
for large time steps under both cases, 
and achieves fourth order convergence for small time steps. 
Furthermore, when $h$ is small enough, 
$\mathsf{e}(h,\epsilon)$ of the PT dynamics 
is smaller than that of the Schr\"odinger dynamics. 
Fig.~\ref{fig:example_2by2_order:a} presents a study on 
how $\mathsf{e}(h,\epsilon)$ depends on $\epsilon$, 
which reveals that by propagating the PT dynamics 
we gain one extra order of accuracy in terms of $\epsilon$. 
This agrees with the theoretical results in Section~\ref{sec:analysis}. 

Next we test GL2 as an example of implicit symplectic methods applied to
the Hamiltonian systems. Fig.~\ref{fig:example_2by2_errors:c}
compares the numerical performances of S-GL2 and PT-Ham-GL2.  For small
$h$, we observe a smaller error using the PT formulation, \ie
$\mathsf{e}(h,\epsilon)$ of S-GL2 is $\Or(h^2/\epsilon^3)$ and
$\mathsf{e}(h,\epsilon)$ of PT-Ham-GL2 is $\Or(h^2/\epsilon^2)$ (see
Fig.~\ref{fig:example_2by2_order:b} for a study on the $\epsilon$
dependence).  This verifies the estimate in Theorem~\ref{thm:GE}. 
Despite that GL2 is a numerically stable scheme with a large time step,
the step size of S-GL2 is constrained by the requirement of the
accuracy, while the step size of PT-Ham-GL2 can be chosen to be 
considerably larger.


\begin{figure} \centering
    \begin{subfigure}[b]{0.48\textwidth}
        \includegraphics[width=\textwidth]{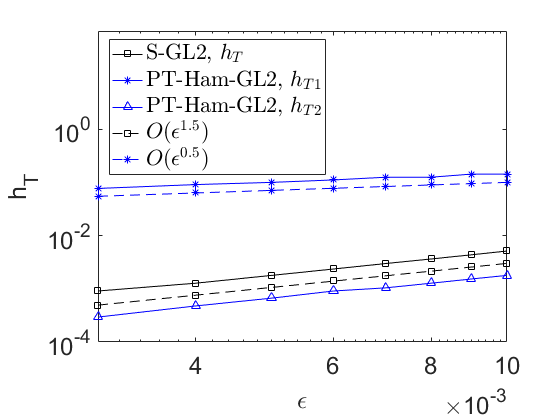}
        \caption{}
        \label{fig:example_2by2_TP:a}
    \end{subfigure} 
    \begin{subfigure}[b]{0.48\textwidth}
        \includegraphics[width=\textwidth]{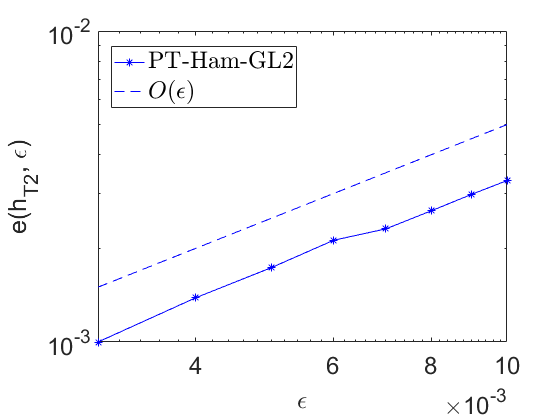}
        \caption{}
        \label{fig:example_2by2_TP:b}
    \end{subfigure}
    \caption{(a) Relationship between the turning points and $\epsilon$ 
    in S-GL2 and PT-Ham-GL2 in the near adiabatic regime 
    of the toy example. 
    (b) Relationship between the magnitude of the plateau of the
    numerical error and $\epsilon$ 
    in PT-Ham-GL2. }
    \label{fig:example_2by2_TP}
\end{figure}

More specifically, let us define the ``turning point'' $h_T$ to be 
the largest time step size when a scheme starts to converge.
Numerically for second order schemes the turning point can be computed as
$$h_T = \arg\max\left\{h\in[h_1,h_2]: \frac{\partial(\log \mathsf{e})}{\partial(\log h)} > 1\right\}$$
where $[h_1,h_2]$ is a suitable interval 
containing the convergence interval of interests. 
In Fig.~\ref{fig:example_2by2_errors:c} 
we mark the turning points in S-GL2 and PT-Ham-GL2,
and study their dependence on $\epsilon$ 
in Fig.~\ref{fig:example_2by2_TP:a}. 
For S-GL2, the convergence starts at $h_T = \Or(\epsilon^{3/2})$. 
For PT-Ham-GL2, a two-stage convergence behavior is observed.
As $h$ decreases, 
the scheme first starts to converge with second order
at a relatively large time step $h_{T1} = \Or(\epsilon^{1/2})$.
This first stage ends at $h = \Or(\epsilon)$ when
$\mathsf{e}(h,\epsilon)$ reaches a plateau with its magnitude being
$\Or(\epsilon)$ (see Fig.~\ref{fig:example_2by2_TP:b}).
Then the second-stage convergence starts at $h_{T2} =
\Or(\epsilon^{3/2})$. 


\begin{figure} \centering
    \begin{subfigure}[b]{0.48\textwidth}    
        \includegraphics[width=\textwidth]{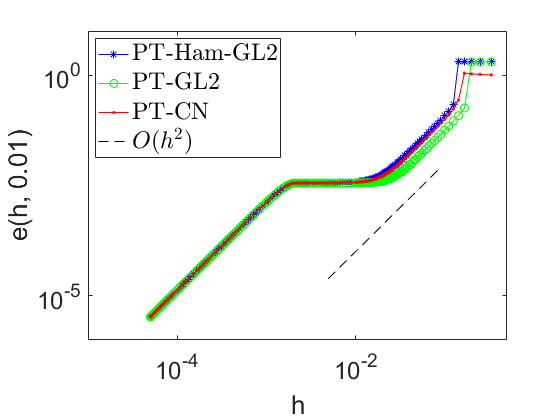}
        \caption{$\epsilon = 0.01$}
        \label{fig:example_2by2_errors:e}    
    \end{subfigure} 
    \begin{subfigure}[b]{0.48\textwidth}    
        \includegraphics[width=\textwidth]{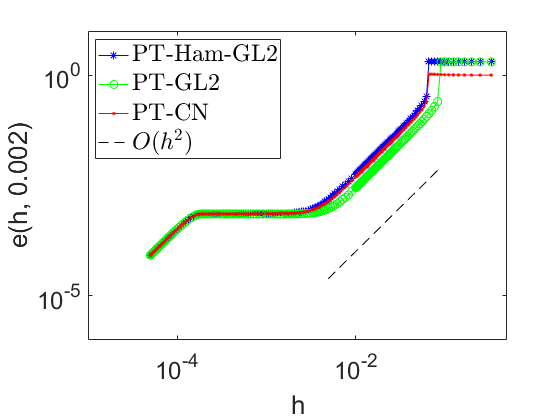}
        \caption{$\epsilon = 0.002$}
        \label{fig:example_2by2_errors:f}    
    \end{subfigure} 
    \caption{Performance of PT-Ham-GL2, PT-GL2 and PT-CN in the 
    near adiabatic regime of the toy example. }
    \label{fig:example_2by2_errors_PT}
\end{figure}

In the end we compare the schemes PT-Ham-GL2, PT-GL2 and PT-CN. 
Although we only justified the behavior of the global error for
PT-Ham-GL2, 
numerical results in Fig.~\ref{fig:example_2by2_errors:e} 
and~\ref{fig:example_2by2_errors:f} 
indicate that there is no essential difference among these methods in
practice.
%

\subsubsection{Beyond adiabatic regime}

\begin{figure} \centering
    \begin{subfigure}[b]{0.48\textwidth}
        \includegraphics[width=\textwidth]{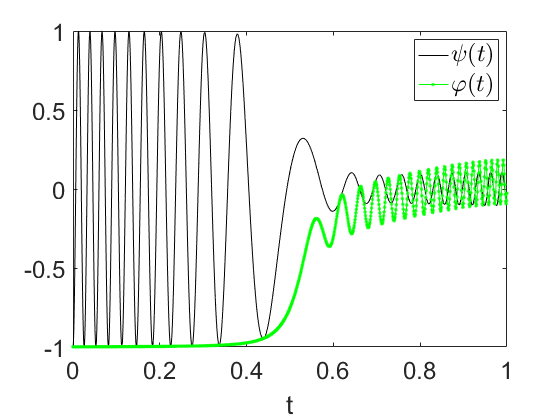}
        \caption{}
        \label{fig:example_2by2_BA_wave:a}
    \end{subfigure} 
    \begin{subfigure}[b]{0.48\textwidth}
        \includegraphics[width=\textwidth]{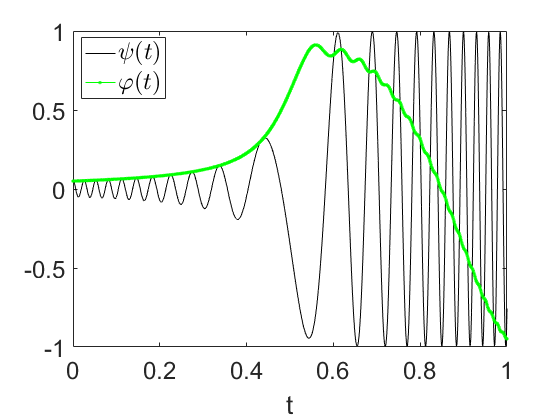}
        \caption{}
        \label{fig:example_2by2_BA_wave:b}
    \end{subfigure}
    \begin{subfigure}[b]{0.48\textwidth}
        \includegraphics[width=\textwidth]{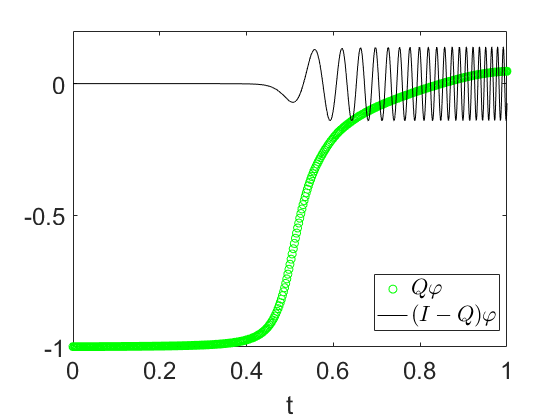}
        \caption{}
        \label{fig:example_2by2_BA_wave:c}
    \end{subfigure}
    \begin{subfigure}[b]{0.48\textwidth}
        \includegraphics[width=\textwidth]{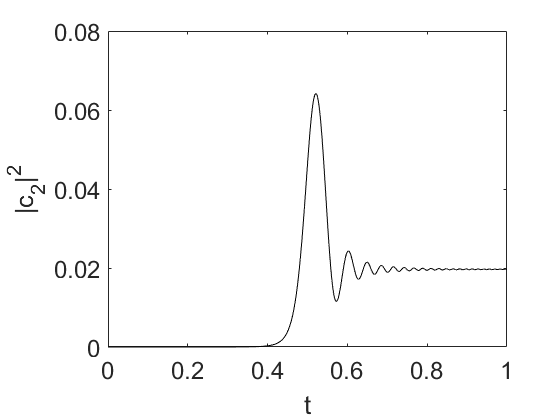}
        \caption{}
        \label{fig:example_2by2_BA_wave:d}
    \end{subfigure}
    \caption{The Schr\"odinger and the PT wave functions 
    beyond the adiabatic regime in the toy example. 
    In all sub-figures, parameters are chosen to be 
    $\epsilon = 0.002,\delta=0.05$, 
    and the solutions are obtained by GL2 with the time step $h = 10^{-6}$. 
    (a)(b) show the first and second entry of the real part of 
    the Schr\"odinger wave function and the PT wave function,
    respectively.
    (c) shows a decomposition of the PT wave function into 
    the two orthogonal eigenspaces 
    (in the sub-figure we only present the real part of the first entry). 
    (d) shows the time evolution of the probability that the eigenstate
    corresponding to $\lambda_2$ is occupied. }
    \label{fig:example_2by2_BA_wave}
\end{figure}

As the value of $\delta$ is reduced, the second eigenstate corresponding
to $\lambda_{2}$ may contribute significantly to the wave function, which
leads to the violation of the adiabatic regime.

Fig.~\ref{fig:example_2by2_BA_wave} investigates the 
Schr\"odinger wave function and the PT wave function with $\epsilon = 0.002, \delta = 0.05$. 
Fig.~\ref{fig:example_2by2_BA_wave:a} and~\ref{fig:example_2by2_BA_wave:b} 
compare the real parts of the Schr\"odinger wave function 
and the PT wave function. 
When $t<t_0=0.5$, the system stays close to the adiabatic regime and the PT
wave function is nearly flat.
However, when $t>t_0$, the PT wave function starts to oscillate as well. 
Fig.~\ref{fig:example_2by2_BA_wave:c} shows an orthogonal decomposition of the 
PT wave function into two orthogonal eigenspaces.
Fig.~\ref{fig:example_2by2_BA_wave:d} shows the evolution of the
probability that the eigenstate corresponding to $\lambda_2(t)$ is occupied, 
which can be computed as $|c_2|^2 = |(\varphi(t), e_2(t))|^2$ 
and $e_2(t)$ is the normalized eigenstate of $H(t)$ corresponding to $\lambda_2(t)$. 
These results confirm that the oscillatory behavior originates from the
excited state corresponding to $\lambda_2$.


\begin{figure} \centering
    \begin{subfigure}[b]{0.48\textwidth}
        \includegraphics[width=\textwidth]{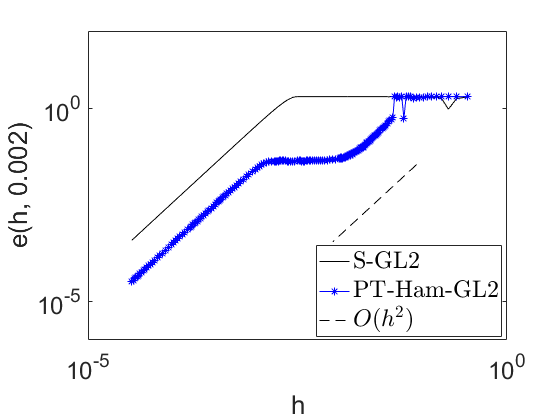}
        \caption{$\delta = 0.07$}
        \label{fig:example_2by2_BA_errors:a}
    \end{subfigure} 
    \begin{subfigure}[b]{0.48\textwidth}
        \includegraphics[width=\textwidth]{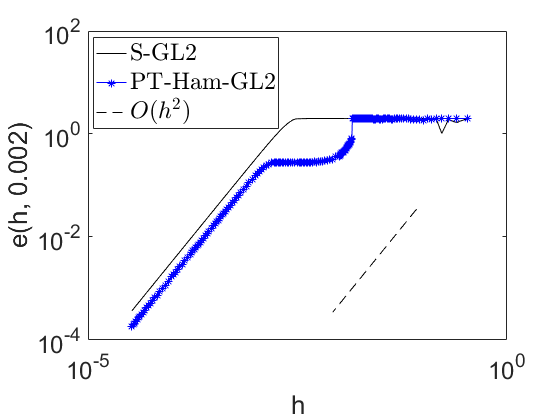}
        \caption{$\delta = 0.05$}
        \label{fig:example_2by2_BA_errors:b}
    \end{subfigure}
    \begin{subfigure}[b]{0.48\textwidth}
        \includegraphics[width=\textwidth]{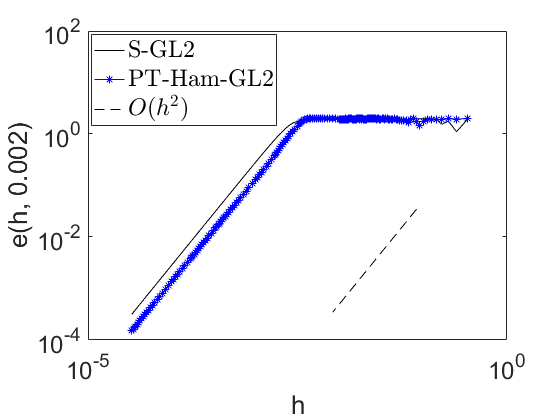}
        \caption{$\delta = 0.03$}
        \label{fig:example_2by2_BA_errors:c}
    \end{subfigure}
    \begin{subfigure}[b]{0.48\textwidth}
        \includegraphics[width=\textwidth]{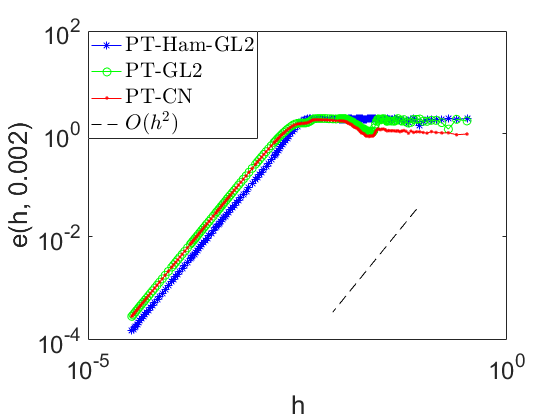}
        \caption{$\delta = 0.03$}
        \label{fig:example_2by2_BA_errors:d}
    \end{subfigure}
    \caption{Numerical errors of different numerical methods 
    beyond the adiabatic regime in the toy example. 
    In all sub-figures $\epsilon = 0.002$. 
    (a)(b)(c) compare the numerical performances 
    between S-GL2 and PT-Ham-GL2 for $\delta = 0.07,0.05,0.03$, respectively. 
    (d) gives a comparison of PT-Ham-GL2, PT-GL2 and PT-CN 
    with $\delta = 0.03$. }
    \label{fig:example_2by2_BA_errors}
\end{figure}


As discussed before, such oscillatory nature in the wave functions 
may increase the computational difficulty and 
require a smaller time step even for the PT dynamics.
Fig.~\eqref{fig:example_2by2_BA_errors} 
compares $\mathsf{e}(h,\epsilon)$  for 
S-GL2, PT-Ham-GL2, PT-GL2 and PT-CN respectively. 
The results confirm that the PT dynamics is always more accurate than
the Schr\"odinger dynamics using the same step size, but the gain
becomes smaller as $\delta$ decreases.

\subsection{Nonlinear Schr\"odinger equation in one dimension}

Next we study the performance of the PT dynamics in a singularly
perturbed nonlinear Schr\"odinger equation in one dimension.
\begin{equation}\label{eqn:exampleNLS}
    \begin{split}
    \I \epsilon \partial_t \psi(x,t)
    &= -\frac{1}{2}\partial^2_x \psi(x,t) + V(x,t)\psi(x,t) + g|\psi(x,t)|^2\psi(x,t), \ \ x \in [0,L] \\
    \psi(x,0) &= \psi_0(x) \\
    \psi(0,t) &= \psi(L,t) {.} 
    \end{split}
\end{equation}
We set $L = 50$, and the external potential is chosen to be a time-dependent
Gaussian function modeling a moving potential well (Fig.~\ref{fig:example_NLS_potential})
\begin{equation}\label{eqn:example_NLS_potential}
    V(x,t) = - \exp(-0.1(x-R(t))^2)
\end{equation}
with a time-dependent center 
\begin{equation}\label{eqn:example_NLS_center}
    R(t) = 25 + 1.5\exp(-25(t-0.1)^2) 
    + \exp(-25(t-0.5)^2) {.}
\end{equation}
Note that $R(t)$ varies on the $\Or(1)$ time scale.

\begin{figure} \centering
    \begin{subfigure}[b]{0.48\textwidth}
        \includegraphics[width=\textwidth]{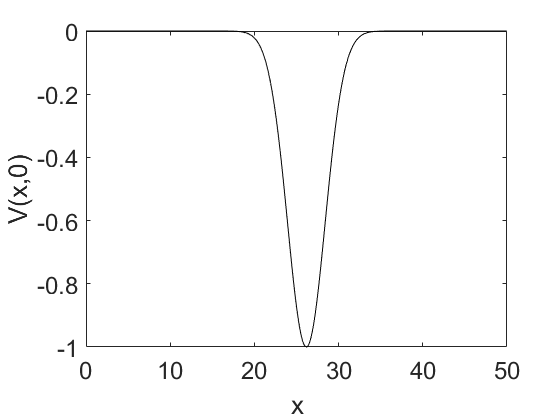}
        \caption{}
        \label{fig:example_NLS_potential:a}
    \end{subfigure} 
    \begin{subfigure}[b]{0.48\textwidth}
        \includegraphics[width=\textwidth]{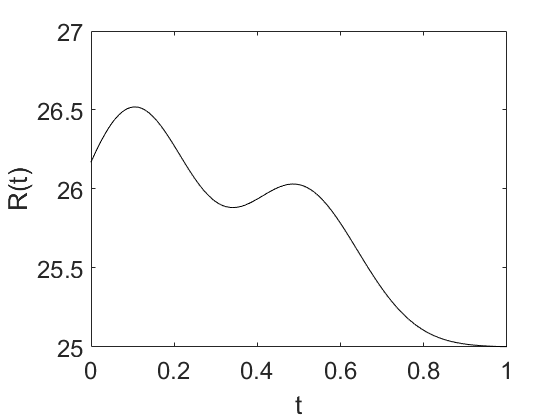}
        \caption{}
        \label{fig:example_NLS_potential:b}
    \end{subfigure}
    \caption{External potential and the time-dependent center for 
    the nonlinear Schr\"odinger equation.}
    \label{fig:example_NLS_potential}
\end{figure}

We use equidistant nodes $x_k = kh_x$ and the second-order finite
difference scheme for spacial discretization, and we fix $h_x = 0.025$. 
Other parameters in this example are chosen to be 
$g = 2.5, T = 1, \epsilon = 0.0025$. 
For the choices of the parameters in the Anderson Mixing, 
the step length $\alpha = 1$, the mixing dimension is 20, 
and the tolerance is $10^{-8}$. 
Fig.~\ref{fig:example_NLS_errors} compares $\mathsf{e}(h,\epsilon)$
of S-GL2, PT-Ham-GL2, PT-GL2 and PT-CN, and confirms the same numerical behavior
as in the toy example. 

%

\begin{figure} \centering
    \begin{subfigure}[b]{0.48\textwidth}
        \includegraphics[width=\textwidth]{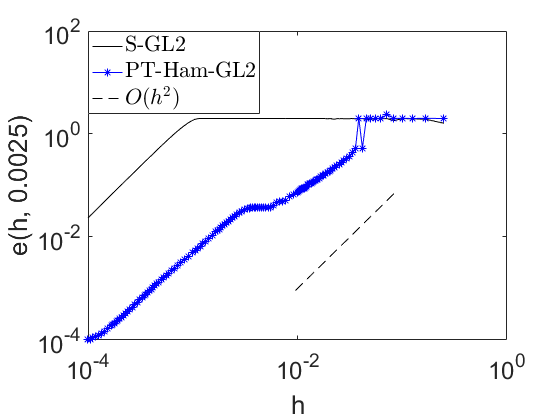}
        \caption{}
        \label{fig:example_NLS_errors:a}
    \end{subfigure} 
    \begin{subfigure}[b]{0.48\textwidth}
        \includegraphics[width=\textwidth]{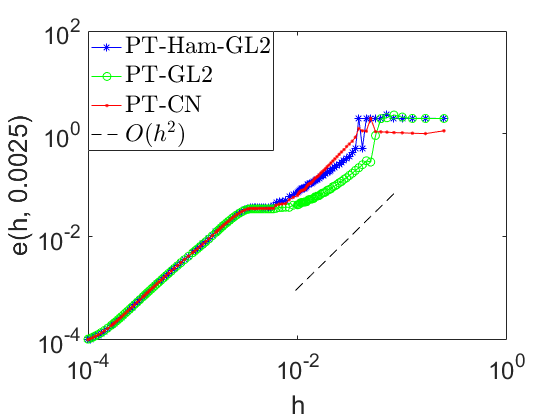}
        \caption{}
        \label{fig:example_NLS_errors:b}
    \end{subfigure}
    \caption{Numerical errors of different numerical methods 
    in the example of the nonlinear Schr\"odinger equation. 
    Parameters are chosen to be $T = 1, \epsilon = 0.0025$. 
    (a) compares S-GL2 and PT-Ham-GL2. 
    (b) compares PT-Ham-GL2, PT-GL2 and PT-CN. }
    \label{fig:example_NLS_errors}
\end{figure}


\begin{figure} \centering
        \includegraphics[width=0.48\textwidth]{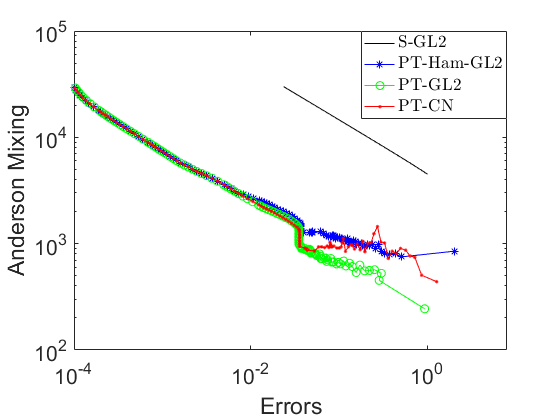}
    \caption{Total numbers of the Anderson mixing versus the numerical error.}
    \label{fig:example_NLS_costs}
\end{figure}

Next we study the computational cost by 
comparing the total number of the Anderson mixing steps versus the
numerical error $\mathsf{e}(h,\epsilon)$ up to $T=1$.
Fig.~\ref{fig:example_NLS_costs} clearly demonstrates that in order to
achieve the same level of accuracy, all the methods propagating the PT
dynamics, including PT-Ham-GL2, PT-GL2 and PT-CN, are much more efficient
than S-GL2. This is valid across the entire range of the step sizes
under study.

\subsection{Time-dependent density functional theory in three dimension}


As the last example, we demonstrate the performance of the PT dynamics
for a benzene molecule driven by an ultrashort laser pulse using the
time-dependent density functional theory (TDDFT).
\REV{The TDDFT equations are
\begin{equation}
  \I \partial_t \Psi(t) = H(t,P)\Psi(t),\quad P(t) = \Psi(t)\Psi^*(t),
  \label{eqn:TDDFT_equations}
\end{equation}
and the corresponding PT-TDDFT equations are
\begin{equation}
  \I \partial_t \Phi(t) = H(t,P)\Phi(t) - \Phi(t)\Phi^*(t)H(t,P)\Phi(t),\quad P(t) = \Phi(t)\Phi^*(t).
  \label{eqn:PT_TDDFT_equations}
\end{equation}
The number of wavefunctions $N$ is 15 for benzene example. Compared to the setup of singularly perturbed equations, here in the sense that the parameter $\epsilon$ is formally set to $1$ in TDDFT equations. However, as will be seen later, the PT dynamics can still result in significant computational advantage. }
The Hamiltonian takes the form
\REV{\begin{equation}
  H(t,P) = -\frac{1}{2}\Delta + V_{\ext}(\vr,t) +
  V_{\mathrm{PP}}(\vr) + V_{\Hxc}[\rho(t)].
  \label{}
\end{equation}}
Here $V_{\mathrm{PP}}$ is the pseudopotential operator due to the
electron-ion interaction, and we use the Optimized Norm-Conserving
Vanderbilt (ONCV) pseudopotential~\cite{Hamann2013} with a kinetic
energy cutoff 
$E_{\mathrm{cut}}=30 \text{ Hartree}$. 
\REV{After spatial discretization, $V_{\mathrm{PP}}$  becomes a matrix independent of the time $t$ and the density matrix $P$.}
$V_{\Hxc}$ is the sum of the Hartree and exchange-correlation
potentials. We use the Perdew-Burke-Ernzerhof
(PBE)~\cite{PerdewBurkeErnzerhof1996} exchange correlation potential 
that depends on the electron density $\rho(t)=\mathrm{diag}[P(t)]$. 
The external potential $V_{\text{ext}}(\vr,t) = \textbf{r} \cdot
\textbf{E}(t)$ is given by a time-dependent electric
field
\begin{equation}
\textbf{E}(t) = \hat{\textbf{k}}E_{\text{max}}\exp \Big[-\frac{(t-t_0)^2}{2a^2}\Big]\sin [\omega(t-t_0)] {,}
\end{equation}
where $\hat{\textbf{k}}$ is a unit vector defining the polarization of the electric field. 
The parameters $a,t_0,E_{\text{max}},\omega$ define
the width, the initial position of the center,
the maximum amplitude of the Gaussian envelope,
and the frequency of the laser, respectively.  In practice $\omega$ and
$a$ are often determined by
the wavelength $\lambda$ and the full width at half maximum (FWHM) pulse width~\cite{RussakoffLiHeVarga2016}, 
\ie $\lambda \omega = 2\pi c$ and $\text{FWHM} = 2a\sqrt{2\log 2}$,
where $c$ is the speed of the light. 
In this example, the peak electric field $E_{\text{max}}$ is 1.0 eV/\r A,
occurring at $t_0 = 15.0$ fs. 
The FWHM pulse width is 6.0 fs,
and the polarization of the laser field is aligned along the $x$ axis 
(the benzene molecule is in $x$-$y$ plane, see 
Fig.~\ref{fig:example_TDDFT_potential:benzene}). 
We consider one relatively slow laser 
with wavelength 800 nm, and another faster
laser with wavelength 250 nm, respectively
(Fig.~\ref{fig:example_TDDFT_potential}). 
The electron dynamics for the first laser is in the near adiabatic
regime, where the system stays near the ground state after the active
time interval of
the laser, while the second laser drives electrons to excited
states.  We implement S-RK4 and PT-CN in the PWDFT package, and
propagate TDDFT to $T = 30.0$ fs.  
For the parameters in the Anderson mixing, the step length $\alpha$ is $0.2$,
the mixing dimension is 10, and the tolerance is $10^{-6}$. 
We measure the accuracy using the dipole moment
$\vD(t):=\Tr[\vr P(t)]$, as well as the energy difference $E(t)-E(0)$
along the trajectory.

\begin{figure}
    \centering
     \begin{subfigure}[b]{0.5\textwidth}
        \includegraphics[width=\textwidth]{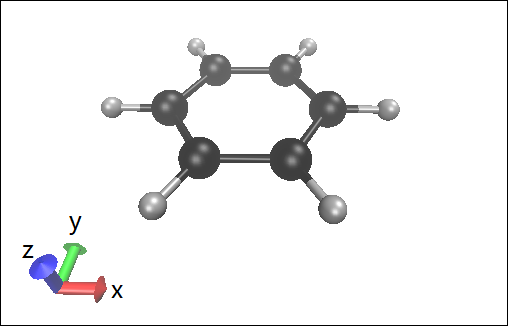}
        \caption{}
        \label{fig:example_TDDFT_potential:benzene}
    \end{subfigure} 

    \centering
       \begin{subfigure}[b]{0.48\textwidth}
       \centering
        \includegraphics[width=\textwidth]{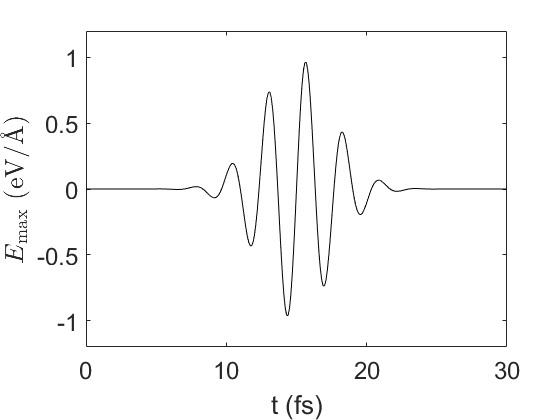}
        \caption{}
        \label{fig:example_TDDFT_potential:a}
    \end{subfigure} 
    \begin{subfigure}[b]{0.48\textwidth}
        \includegraphics[width=\textwidth]{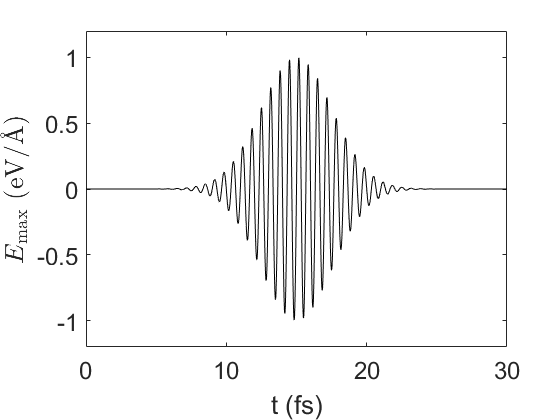}
        \caption{}
        \label{fig:example_TDDFT_potential:b}
    \end{subfigure}
  \caption{
  (a) The benzene molecule. The direction of the external electric field is 
  along the x-axis. 
  This figure is generated by VMD package~\cite{HumphreyDalkeSchulten1996}. 
  (b)(c) The intensity of the electric field. 
  The peak electric field $E_{\text{max}}$ is 1.0 eV/\r A,
occurring at $t_0 = 15.0$ fs, and the FWHM pulse width is 6.0 fs. 
The wavelength is 800 nm in (b), and 250 nm in (c). 
  }\label{fig:example_TDDFT_potential}
\end{figure}


\begin{figure}
    \centering
    \begin{subfigure}[b]{0.48\textwidth}
        \includegraphics[width=\textwidth]{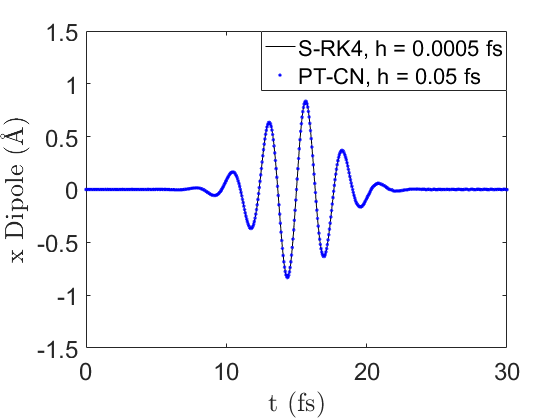}
        \caption{}
        \label{fig:example_TDDFT_results:a}
    \end{subfigure} 
    \begin{subfigure}[b]{0.48\textwidth}
        \includegraphics[width=\textwidth]{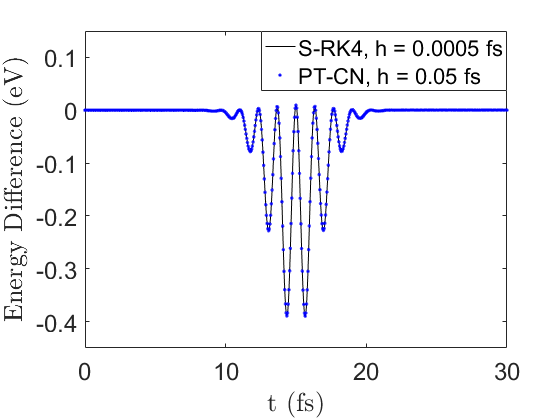}
        \caption{}
        \label{fig:example_TDDFT_results:b}
    \end{subfigure}
  \caption{(a) Dipole moment along the x-direction and (b) total energy
  difference with the 800 nm laser.}
 \label{fig:example_TDDFT_results}
\end{figure}

\begin{figure}
    \centering
    \begin{subfigure}[b]{0.48\textwidth}
        \includegraphics[width=\textwidth]{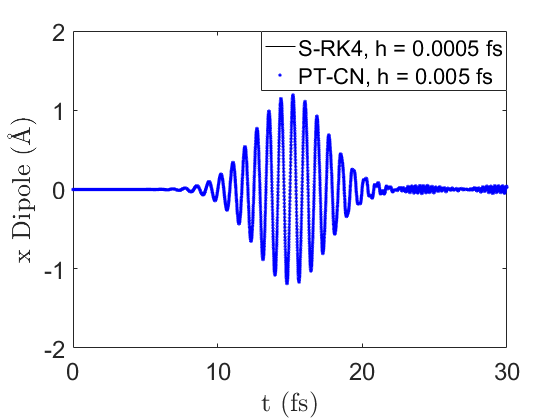}
        \caption{}
        \label{fig:example_TDDFT_results2:a}
    \end{subfigure} 
    \begin{subfigure}[b]{0.48\textwidth}
        \includegraphics[width=\textwidth]{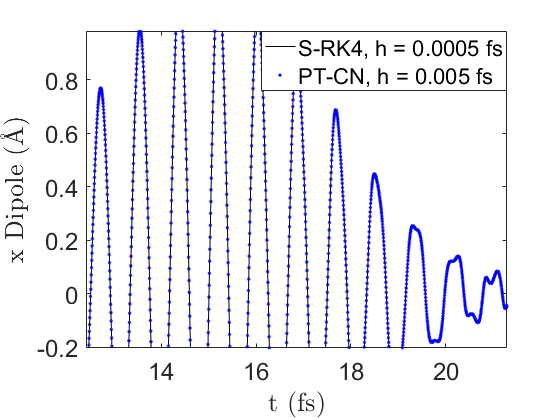}
        \caption{}
        \label{fig:example_TDDFT_results2:b}
    \end{subfigure}
    \begin{subfigure}[b]{0.48\textwidth}
        \includegraphics[width=\textwidth]{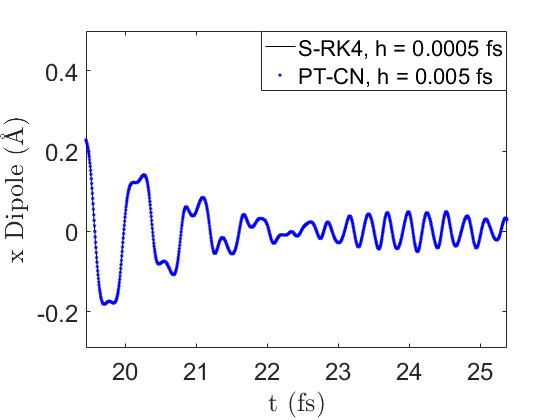}
        \caption{}
        \label{fig:example_TDDFT_results2:c}
    \end{subfigure}
    \begin{subfigure}[b]{0.48\textwidth}
        \includegraphics[width=\textwidth]{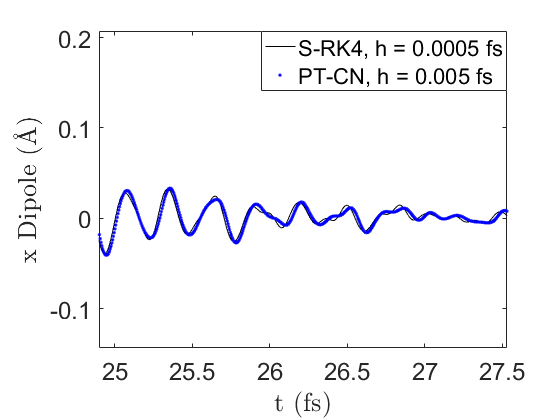}
        \caption{}
        \label{fig:example_TDDFT_results2:d}
    \end{subfigure}
    \begin{subfigure}[b]{0.48\textwidth}
        \includegraphics[width=\textwidth]{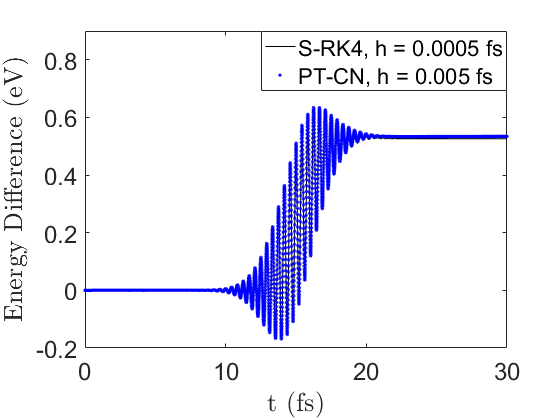}
        \caption{}
        \label{fig:example_TDDFT_results2:e}
    \end{subfigure}
    \begin{subfigure}[b]{0.48\textwidth}
        \includegraphics[width=\textwidth]{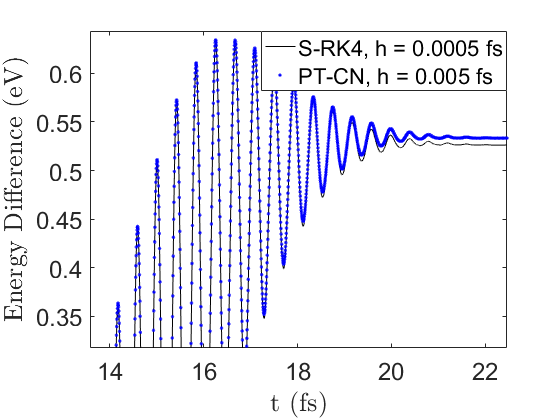}
        \caption{}
        \label{fig:example_TDDFT_results2:f}
    \end{subfigure}
  \caption{(a) Dipole moment along the x-direction and (e) total energy
  difference with the 250 nm laser, with zoom-in views provided in (b)(c)(d)(f).}
%
  \label{fig:example_TDDFT_results2}
\end{figure}

Figure~\ref{fig:example_TDDFT_results} shows the numerical results 
for the 800 nm laser using S-RK4 with a step size 0.0005 fs 
and PT-CN with a step size 0.05 fs. 
In this case, the system stays near the ground state after 
the active time interval of the laser. 
After 25.0 fs, the total energy for S-RK4 only increases 
by $2.00\times 10^{-4}$ eV, and hence we may use the results from S-RK4
as our benchmark.
We remark that S-RK4 becomes unstable at large time step sizes. 
Even when increasing the time step to be 0.001 fs, S-RK4 blows up 
within 100 time steps.
We observe that PT-CN agrees perfectly with S-RK4 
in terms of the dipole moment along the $x$ direction, and the
total energy difference.
After 25.0 fs, the total energy is nearly constant 
and only slightly increases by $2.44\times 10^{-4}$ eV compared to that
of the initial state.

Since the computational cost of TDDFT calculations is mainly dominated
by the cost of applying the Hamiltonian matrix to wave functions, 
we measure the numerical efficiency using the number of such
matrix-vector multiplications.
Although PT-CN requires more matrix-vector multiplications 
in each time step, 
the total number of matrix-vector multiplications is
still significantly reduced due to the larger time step size,
and PT-CN usually achieves a significant speedup.
More specifically, in this case, during the time interval for which the
laser is active (from 5.5 fs to 24.5 fs), 
the average number of matrix-vector multiplications 
in each PT-CN time step is 12.6, 
and the total number of matrix-vector multiplications in the simulation is 4798. 
On the other hand,
the number of matrix-vector multiplications in each S-RK4 time step is 4, 
and the total number of matrix-vector multiplications 
during this period using time step 0.0005 fs 
is 152000. 
Hence the overall speedup of PT-CN over RK4 is $31.7$.

Figure~\ref{fig:example_TDDFT_results2} shows the numerical results for
the 250 nm laser.  In this case, the laser carries more energy and hence
a significant amount of electrons can reach the excited states.
According to the S-RK4 benchmark, 
the total energy of the system increases by 0.5260 eV after 25.0 fs. 
Furthermore, the dipole moment along the $x$ direction oscillates more
strongly due to the excitation. PT-CN needs to adopt a smaller time 
step size 0.005 fs, and still gives a very good approximation to the
electron dynamics compared to S-RK4, 
For the dipole moment, 
PT-CN results match very well with S-RK4 benchmark during
(Fig.~\ref{fig:example_TDDFT_results2:b}) and after
(Fig.~\ref{fig:example_TDDFT_results2:c}
and~\ref{fig:example_TDDFT_results2:d}) the active time interval of the laser.
The total energy obtained by PT-CN matches very well with 
that in S-RK4 benchmark during the active interval
and stays at a constant level with an average increase of 0.5340 eV 
by the end of the simulation (Fig.~\ref{fig:example_TDDFT_results2:e}
and~\ref{fig:example_TDDFT_results2:f}). 
In this case, PT-CN slightly overestimates the total energy 
after the laser's action by $7.96\times 10^{-3}$ eV. 

For the computational costs within the period from 5.5 fs to 24.5 fs, 
the total number of 
matrix-vector multiplications is still 152000 for S-RK4. 
The average number of matrix-vector multiplications 
in each PT-CN time step is 7.5 due to the reduced step size, 
and the total number of matrix-vector multiplications is 28610. 
Therefore in this case PT-CN achieves 5.3 times speedup over S-RK4. 

\begin{table}
  \centering
  \begin{tabular}{cc|cc|cc}
  \hline Method & h (fs) & AEI (eV) & AOE (eV) & MVM & Speedup\\\hline
  S-RK4 & 0.0005 & 0.5260  & /       & 152000 & /    \\
  PT-CN & 0.005  & 0.5340  & 0.0080  & 28610  & 5.3  \\
  PT-CN & 0.0065 & 0.5347  & 0.0087  & 22649  & 6.7  \\
  PT-CN & 0.0075 & 0.5362  & 0.0102  & 21943  & 6.9  \\
  PT-CN & 0.01   & 0.5435  & 0.0175  & 15817  & 9.6  \\
  PT-CN & 0.02   & 0.5932  & 0.0672  & 12110  & 12.6 \\\hline
  \end{tabular}
  \caption{Accuracy and efficiency of PT-CN for the electron dynamics
  with the 250 nm laser compared to S-RK4.
  The accuracy is measured using 
  the average energy increase (AEI) after 25.0 fs and
  the average overestimated energy (AOE) after 25.0 fs.
  The efficiency is measured using 
  the total number of matrix-vector multiplications (MVM) 
  during the time interval from 5.5 fs to 24.5 fs, 
  and the computational speedup.
  }\label{tab:Example_TDDFT_EnergyDifference}
\end{table}


We remark that even the electron dynamics is beyond the adiabatic
regime, PT-CN can still be stable with a larger time step.
Table~\ref{tab:Example_TDDFT_EnergyDifference} measures the accuracy of
PT-CN with $h=$ 0.005 fs, 0.0065 fs, 0.0075 fs, 
0.01 fs and 0.02 fs, respectively. 
We find
that the number of matrix-vector multiplications systematically reduces
as the step size increases. When the step size is 0.02 fs, the speed up
over S-RK4 is $12.6$, and this is at the expense of overestimating
the energy by $0.0672$ eV after the active interval of the laser.  
Hence one can use PT-CN to quickly study the electron dynamics with a
large time step, while this is not possible using an explicit scheme
like S-RK4.




\section{Conclusion}\label{sec:conclusion}

Quantum dynamics can be equivalently written in terms of the
Schr\"odinger equation for the wave function, and the von Neumann
equation for the density matrix. However, the Schr\"odinger dynamics may
require a very small time step in numerical simulation due to the
non-optimal gauge choice. In this paper, we propose to close this gap by
identifying the optimal gauge choice, which is obtained from the
parallel transport formulation.  The solution of the resulting parallel
transport (PT) dynamics can be significantly less oscillatory to that of the
Schr\"odinger dynamics, especially in the near adiabatic regime.  The PT
dynamics is suitable to be combined with implicit time integrators,
which allows the usage of large time steps even when the spectral radius
of the Hamiltonian is large, and/or when $\epsilon$ is small. Although
our global error analysis only applies to the Hamiltonian form of the PT
dynamics with symplectic integrators and 
a relatively small time step, our numerical results
indicate that the PT dynamics can be effectively discretized with more
general numerical schemes and with much larger time steps. The
mathematical understanding of the behavior with a large time step is our
future work. 
Combining the PT dynamics with numerical schemes other than the
Runge-Kutta methods and the linear multistep methods, as well as  more
detailed numerical studies of the PT dynamics for the time-dependent
density functional theory calculations are also under progress.

\section*{Acknowledgments} 

This work was partially supported by the National Science Foundation
under Grant No. 1450372, No. DMS-1652330  (D. A. and L. L.), and by the Department of
Energy under Grant No. DE-SC0017867, No. DE-AC02-05CH11231 (L. L.). 
We thank the National Energy Research Scientific Computing (NERSC)
center and the Berkeley Research Computing (BRC) program at the
University of California, Berkeley for making computational resources
available.  We thank Stefano Baroni, Roberto Car, Weile Jia, Christian
Lubich, and Lin-Wang Wang for helpful discussions.

\bibliographystyle{siam}
\bibliography{ptref}

\end{document}